\documentclass[reqno,11pt]{amsart}
\usepackage[utf8]{inputenc}
\usepackage[T1]{fontenc}
\usepackage[ngerman, textwidth=35mm, textsize=tiny]{todonotes}
\usepackage{lmodern}
\usepackage[english]{babel}
\usepackage{amsmath,mathtools,amsthm,amssymb,amsfonts,amsrefs}
\usepackage{mathrsfs}
\usepackage{graphicx}
\usepackage[shortlabels]{enumitem}
\usepackage{pgfplots}
\usepackage{esvect}
\usepackage{pdfpages}
\usepackage[all]{xy}
\usepackage{stix}
\usepackage{cancel}
\usepackage{hyperref}
\usepackage{nicefrac}
\usepackage{overpic}

\theoremstyle{plain}
\newtheorem{satz}{Theorem}[section]
\newtheorem{lem}[satz]{Lemma}

\newtheorem{koro}[satz]{Corollary}
\newtheorem{prop}[satz]{Proposition}

\theoremstyle{definition}
\newtheorem{defi}[satz]{Definition}
\newtheorem{bem}[satz]{Remark}
\newtheorem{bsp}[satz]{Example}
\newtheorem{ass}[satz]{Assumption}

\makeatletter
\g@addto@macro\bfseries{\boldmath}
\makeatother

\hypersetup{
	colorlinks,
	linkcolor={red!50!black},
	citecolor={blue!50!black},
	urlcolor={blue!80!black}
}
\definecolor{darkgreen}{rgb}{0.0, 0.5, 0.0}
\definecolor{darkred}{rgb}{0.5, 0, 0.0}

\newcommand{\id}{\operatorname{id}}
\newcommand{\JBeps}{J_B^{(\varepsilon)}}
\newcommand{\vol}{\widehat{\mathrm{vol}}}
\newcommand{\ngm}{\text{\ng}}

\newcommand{\dt}[1]{\ensuremath{\operatorname{d}\!#1}}
\newcommand{\R}{\mathbb{R}}
\newcommand{\abs}[1]{\ensuremath{\left\vert #1\right\vert}}
\newcommand{\spann}[1]{\ensuremath{\left\langle #1\right\rangle}}
\newcommand{\norm}[1]{\ensuremath{\left\| #1\right\|}}
\newcommand{\maxd}{\operatorname{max}}
\newcommand{\dom}{\operatorname{dom}}
\newcommand{\res}{\operatorname{res}}
\newcommand{\Rea}{\operatorname{Re}}

\newcommand{\APS}{\mathrm{APS}}
\newcommand{\cc}{\mathrm{cc}}
\renewcommand{\c}{\mathrm{c}}
\newcommand{\SM}[2]{\ensuremath{\spann{#1,#2}_{SM}}}

\makeatletter
\newcommand{\Mat}[1] {
	\begin{pmatrix}
		\Mat@r #1;\@bye;\Mat@r
	\end{pmatrix}
}

\def\Mat@r #1;{\@bye #1\Mat@z\@bye\Mat@s #1,\@bye, }%
\def\Mat@s #1,{#1\Mat@t }%
\def\Mat@t #1,{\@bye #1\Mat@y\@bye\@firstofone {&#1}\Mat@t}%
\def\Mat@y #1\Mat@t{\\ \Mat@r }
\def\Mat@z #1\Mat@r {}
\def\@bye #1\@bye {}
\makeatother

\title[Cauchy problem under non-local boundary conditions]{The Cauchy problem for Lorentzian Dirac operators under non-local boundary conditions}

\author{Christian Bär}
\address{Department of Mathematics, University of Potsdam, 14476 Potsdam, Germany}
\email{cbaer@uni-potsdam.de}

\author{Penelope Gehring}
\address{Max-Planck Institute for Gravitational Physics, 14476 Potsdam, Germany}
\email{penelope.gehring@aei.mpg.de}

\begin{document}
\maketitle

\begin{abstract}

Non-local boundary conditions, such as the Atiyah-Patodi-Singer (APS) conditions, for Dirac operators on Riemannian manifolds are well under\-stood while not much is known for such operators on spacetimes with timelike boundary.
We define a class of Lorentzian boundary conditions that are local in time and non-local in the spatial directions and show that they lead to a well-posed Cauchy problem for the Dirac operator.
This applies in particular to the APS conditions imposed on each level set of a given Cauchy temporal function.
\end{abstract}

\section*{Introduction}

In General Relativity spacetime is mathematically modeled by a, generally curved, Lorentzian manifold.
To describe wave propagation we must be able to solve initial value problems for wave equations.
If the spacetime has no boundary, this is always possible if the underlying manifold is \emph{globally hyperbolic}.
This means that there exist Cauchy hypersurfaces on which initial values can be imposed.
Well-posedness of the Cauchy problem for linear second order equations can e.g.\ be found in \cite{Waves}.
It is then not hard to also deduce it for Dirac equations, which are of first order.

Assuming for a moment that the spacetime is spatially compact, i.e.\ the Cauchy hypersurfaces are compact manifolds without boundary, we can consider the spacetime region between two (smooth spacelike) Cauchy hypersurfaces, one lying in the future of the other.
The Cauchy hypersurfaces being closed Riemannian manifolds, their Dirac operators are elliptic and, in the selfadjoint case, spectral projectors make sense.
We can therefore impose the famous Atiyah-Patodi-Singer (APS) boundary conditions on both Cauchy hypersurfaces.
Bär and Strohmaier showed that these spatial boundary conditions turn the Dirac operator into a Fredholm operator and gave a geometric index formula (\cites{BS1,BS3}).
As an application, they computed the chiral anomaly in algebraic quantum field theory on curved spacetimes in \cite{BS2}.
Bär and Hannes~\cite{BH} investigated to what extend these boundary conditions can be replaced by more general ones and how the index then changes.
An analogous result to~\cite{BS1} was obtained by Shen and Wrochna~(\cite{SW}) for asymptotically static spacetimes with only one spacelike Cauchy boundary hypersurface.

The Anti-de Sitter (AdS) and the asymptotically AdS spacetimes became increasingly important in recent years - especially in the context of studying the properties of Green-hyperbolic operators like the wave, the Klein-Gordon, or the Dirac operator; see for example \cite{Bachelot,Holzegel, Wrochna,Vasy}.
Green hyperbolic here means that advanced and retarded Green's operators exist; this can be deduced from well-posedness of the Cauchy problem.
The manifolds considered in these results are spacetimes with timelike boundary, i.e.\ the Cauchy hypersurfaces have a boundary themselves.
Further studies of Green-hyperbolic operators on such spacetimes were accomplished, for example, in \cite{DDF,DDL,DFH,GM,MG}.
These results are concerned with local boundary conditions.
Since analytically reasonable local boundary conditions do not always exist for first-order operators, it is necessary to study global boundary conditions such as APS boundary conditions.

APS boundary conditions cannot be imposed on the boundary of the spacetime itself because this boundary is Lorentzian rather than Riemannian.
To overcome this difficulty, Drago, Große and Murro suggested in \cite{DGM} to choose a time function and to impose APS on each level set.
They stated well-posedness of the Cauchy problem for the spinorial Dirac operator under slicewise APS conditions assuming a few further conditions on the geometry of the spacetime.
However, it seems that there is a gap in their existence proof.

We also take this local-in-time, global-in-space approach and fix a time function.
The boundary conditions which we impose on the level sets need not be APS but APS is included in our class of admissible boundary conditions.
Another prominent example consists of the transmission conditions - this may potentially allow for cut-and-paste arguments.

In Section~\ref{EllipticBoundary}, we recall the most important facts on elliptic boundary problems for Dirac-type operators on compact Riemannian manifolds with boundary.
This theory was initiated by Atiyah, Patodi, and Singer in their famous paper \cite{APS}.
Here they introduced the APS boundary conditions.
There is a huge literature on this topic.
We follow the approach laid out by Bär and Ballmann in \cites{BB1, BB2} because it describes all possible boundary conditions systematically.

This describes what can be done on a fixed spacelike Cauchy hypersurface.
We then discuss families of boundary value problems in Section~\ref{ConFunCal}.
We need to study such families as the time function sweeps out our spacetime.

Next,we recall the most important facts of spin geometry and the spinorial Dirac operator on spacetimes with timelike boundary in Section~\ref{sec:preliminaries}.
We define what we mean by an admissible boundary conditions.

Section~\ref{sec:IBVP} is the core of the present paper.
Here we prove well-posedness of the Cauchy problem for the spinorial Dirac operator under admissible boundary conditions, see Theorem~\ref{main}.

We apply these results in Section~\ref{sec:propagation} and study the support of the solutions to the Cauchy problem.
The result essentially says that a wave propagates with the speed of light at most.
There is an interesting violation of this principle, however.
As soon as the wave hits the boundary somewhere, the \emph{whole} boundary radiates off instantenously (w.r.t.\ the given time function).
This is due to the global nature of the boundary conditions and violates the causal principle that no signal should propagate fast than with the speed of light.
We show by example that this effect really occurs.

Then we use well-posedness together with finite propagation speed to construct advanced and retarded Green's operators.

We conclude the paper by discussing some examples in Section~\ref{sec:examples}.
In particular, we introduce Grassmannian boundary conditions which are special admissible boundary conditions, but the assumptions are easier to check in practice, see Theorem~\ref{main:Grassmannian}.
APS boundary conditions fall into this class.

\bigskip

\paragraph*{\emph{Acknowledgments.}} We would like to thank Lashi Bandara, Nicol\'{o} Drago, Nadine Große, Sebastian Hannes, Rubens Longhi, Jan Metzger, Simone Murro, Miguel S\'anchez, and Mehran Seyedhosseini for helpful discussions.

This research is supported by the International Max Planck Research School for Mathematical and Physical Aspects of Gravitation, Cosmology and Quantum Field Theory and by the focus program on Geometry at Infinity (SPP~2026 funded by Deutsche Forschungsgemeinschaft).

\section{Boundary conditions for Dirac-type operators on Riemannian manifolds}
\label{EllipticBoundary}

In this preliminary section we collect the relevant facts about boundary conditions for Dirac-type operators.
Most of the material is well known which is why we just display the results without proof.
We mostly follow \cite{BB1,BB2}.
In \cite{BB3} the theory is developed for general elliptic first-order operators, but we will not need this level of generality.

\medskip
Throughout this section let $(\Sigma,g)$ be a compact Riemannian manifold with smooth boundary $\partial\Sigma$.
We denote its unit conormal field by $\ngm$.
Furthermore, let $(E,h_E)\to \Sigma$ be a Hermitian vector bundle.
Recall that a differential operator $D\colon C^\infty(\Sigma,E)\to C^\infty(\Sigma,E)$ of order one is said to be of \emph{Dirac type} if its principal symbol $\sigma_D$ satisfies the \emph{Clifford relations}
	\begin{align*}
		\sigma_D(\zeta)\sigma_D(\xi)+\sigma_D(\xi)\sigma_D(\zeta)=-2g(\zeta,\xi)\cdot \id_{E_x},
	\end{align*}
for all $x\in \Sigma$ and $\zeta,\xi\in T^\ast_x \Sigma$.

\medskip
The Riemannian spin Dirac operator acting on spinor fields is an important example of a Dirac-type operator.
More generally, Dirac operators in the sense of Gromov and Lawson as introduced in \cite{GL,LM} are of Dirac-type.

\medskip
Dirac-type operators are elliptic, i.e.\ $\sigma_D(\zeta)$ is invertible for all $x\in \Sigma$ and all $\zeta\in T_x ^\ast \Sigma\setminus\{0\}$, since by the Clifford relations the inverse is explicitly given as
\begin{align*}
	\sigma_D(\zeta)^{-1}=-\abs{\zeta}_g^{-2}\sigma_D(\zeta).
\end{align*}
Before discussing boundary conditions for Dirac-type operators, let us spend some time on introducing the maximal and minimal domain of an operator and the range of the restriction map to the boundary mapping from these domains.

\medskip
The \emph{maximal extension} $D_{\maxd}$ is defined to be the distributional extension of $D$, restricted to the \emph{maximal domain}
\[
\dom(D_{\maxd}):=\{\psi\in L^2(\Sigma,E);~D\psi\in L^2(\Sigma,E)\}.
\]
The maximal domain together with the \emph{graph norm} defined by
\[
\norm{\cdot}_D^2:=\norm{D\cdot}^2_{L^2}+\norm{\cdot}^2_{L^2}
\]
forms a Banach space.

The \emph{minimal extension} $D_{\operatorname{min}}$ is the closure of $D$ when $D$ is given the domain $C^\infty_{\cc}(\Sigma,E)$.
Here $C^\infty_{\cc}(\Sigma,E)$ is the space of smooth sections with compact support contained in the interior of $\Sigma$, i.e.\ $\operatorname{supp}\psi\cap \partial \Sigma=\emptyset$ for $\psi\in C^\infty_{\cc}(\Sigma,E)$.
Hence the domain of $D_{\operatorname{min}}$, called the \emph{minimal domain}, is the closure of $C^\infty_{\cc}(\Sigma,E)$ with respect to the graph norm of $D$.
The minimal domain $\operatorname{dom}(D_{\operatorname{min}})$ equipped with the graph norm $\norm{\cdot}_D$ is a Banach space as well.

\begin{defi}\label{boundaryOp}
	A first-order differential operator $A:C^\infty(\partial\Sigma,E|_{\partial\Sigma})\to C^\infty(\partial\Sigma,E|_{\partial\Sigma})$ is called a \emph{boundary operator} for $D$ if its principal symbol is given by
	\[
	\sigma_A(x,\zeta)=\sigma_D(x,\ngm(x))^{-1}\circ\sigma_D(x,\zeta)
	\]
	for all $x\in\partial\Sigma$ and $\zeta\in T^\ast_x\partial \Sigma$.
\end{defi}
\begin{bem}
	\begin{enumerate}[label=(\alph*), wide, labelwidth=!, labelindent=0pt]
		\item
		The boundary operator is unique up to zero-order terms.
		It can be chosen such that it is selfadjoint and that it anticommutes with $\sigma_D(\ngm)$, see Lemma~2.2 in \cite{BB2}.
		For Dirac operators in the sense of Gromov and Lawson there is a natural choice of a selfadjoint boundary operator $A$ which anticommutes with $\sigma_D(\ngm)$.
		This operator has a lower order term involving the mean curvature of the boundary, see for example \cite{BB2}.
		For the spin Dirac operator this can also be seen directly, see~\eqref{eq:closedDirac}.
		\item
		Since $\partial\Sigma$ is compact and without boundary, the boundary operator $A$, if chosen selfadjoint, has discrete real spectrum.
		The case of noncompact $\partial \Sigma$ has been studied in \cite{GN} under suitable geometric conditions.
	\end{enumerate}
\end{bem}
\begin{ass}
For the rest of this section, we assume that $D$ is a selfadjoint Dirac-type operator and that $A$ is a selfadjoint boundary operator for $D$ which anticommutes with $\sigma_D(\ngm)$.
\end{ass}

Let $\chi^+(A)\colon L^2(\partial\Sigma,E|_{\partial\Sigma})\to L^2(\partial\Sigma,E|_{\partial\Sigma})$ and $\chi^-(A)\colon L^2(\partial\Sigma,E|_{\partial\Sigma})\to L^2(\partial\Sigma,E|_{\partial\Sigma})$ be the spectral projections onto the spectral subspaces corresponding to the positive and the nonpositive eigenvalues of $A$, respectively.
These projections are pseudo-differential operators of order zero and therefore
\[
\chi^\pm(A)H^s(\partial\Sigma,E|_{\partial\Sigma})
\]
are closed subspaces of the Sobolev spaces $H^s(\partial\Sigma,E|_{\partial\Sigma})$ for all $s\in\mathbb{R}$.
\begin{defi}
	We define the \emph{check space} corresponding to the boundary operator $A$ as
	\[
	\check{H}(A):=\chi^-(A)H^{\frac 12}(\partial\Sigma,E|_{\partial\Sigma})\oplus \chi^+(A)H^{-\frac 12}(\partial\Sigma,E|_{\partial\Sigma}),
	\]
	with norm
	\[
	\norm{\psi}^2_{\check{H}(A)}:=\norm{\chi^-(A)\psi}^2_{H^{\frac 12}}+\norm{\chi^+(A)\psi}^2_{H^{-\frac 12}}.
	\]
\end{defi}
The check space arises naturally as the image of the extension of the trace map $R\colon C^\infty(\Sigma,E)\to C^\infty(\partial \Sigma,E|_{\partial\Sigma}),~\psi\mapsto \psi|_{\partial\Sigma}$ to the maximal domain of $A$.
More precisely:
\begin{satz}[Theorem~6.7 in \cite{BB1}] \label{BBMain}
	The following holds:
	\begin{enumerate}[label=(\arabic*), labelwidth=!, labelindent=2pt, leftmargin=21pt]
		\item $C^\infty(\Sigma,E)$ is dense in $\operatorname{dom}(D_{\maxd})$ with respect to $\norm{\cdot}_D$.
		\item The trace map extends uniquely to a continuous surjection
		\[
		R\colon \operatorname{dom}(D_{\maxd})\to \check{H}(A)
		\]
		with kernel $\operatorname{ker}R=\dom(D_{\operatorname{min}})$.
		In particular, $R$ induces an isomorphism
		\[
		\check{H}(A)\cong \dfrac{\dom(D_{\maxd})}{\dom(D_{\operatorname{min}})}.
		\]
		\item For all $\phi,\psi\in\operatorname{dom}(D_{\maxd})$
		\begin{align}\label{RiemGreen}
			\int_\Sigma h_E(D_{\maxd}\phi,\psi)-h_E(\phi,D_{\maxd}\psi)\dt{\mu}_{\Sigma}=-\int_{\partial\Sigma} h_E(\sigma_D(\ngm)R\phi,R\psi)\dt{\mu}_{\partial\Sigma}.
		\end{align}
		\item $H^1(\Sigma,E)\cap \dom(D_{\maxd})=\{\psi\in \operatorname{dom}(D_{\maxd});~R\psi\in H^{\frac 12}(\partial\Sigma,E|_{\partial\Sigma})\}$.

	\end{enumerate}
\end{satz}
Note that the RHS of \eqref{RiemGreen} is well defined because $\sigma_D(\ngm)$ maps $\check{H}(A)$ to $\check{H}(-A)$, since it anticommutes with $A$.
Theorem~\ref{BBMain} gives us the necessary tools to define boundary conditions.

\begin{defi}\label{def:bound}
	A \emph{boundary condition} is a closed linear subspace $B\subseteq \check{H}(A)$.
	 The domains of the associated operators are
	\begin{align*}
		\dom (D_{\maxd,B})&:=\{\psi\in \dom(D_{\maxd});~R\psi\in B\},~~\text{and} \\
		\dom(D_B)&:=\{\psi\in\operatorname{dom}(D_{\maxd})\cap H^1(\Sigma,E);~R\psi\in B\}.
	\end{align*}
\end{defi}

We have a 1-1 relation between boundary conditions and closed extensions of $D$ between the minimal and the maximal extension.
Moreover, $(\dom(D_{\maxd,B}),\norm{\cdot}_D)$ is a Banach space for any boundary condition $B$.
A boundary condition $B$ satisfies $B\subseteq H^{\frac 12}(\partial\Sigma,E|_{\partial\Sigma})$ if and only if $D_B=D_{\maxd,B}$.
Motivated by \eqref{RiemGreen}, the \emph{boundary condition adjoint to} $B$ is defined as
\[
B^\ast
:=
\Big\{\phi\in\check{H}(A);~\int_{\partial\Sigma}h_E(\sigma_D(\ngm)\psi,\phi)\dt{\mu}_{\partial\Sigma}=0~\forall\,\psi\in B\Big\}.
\]
We call a boundary condition $B$ \emph{selfadjoint} if $B^\ast=B$.
By Subsection~7.2 in \cite{BB1}, the domain of the adjoint of $D_{\maxd,B}$ is given by
	\[
	\dom((D_{\maxd,B})^\ast)
	=
	\{\psi\in\operatorname{dom}(D_{\maxd});\,\psi|_{\partial\Sigma}\in B^\ast\}=\dom(D_{\maxd,B^\ast}).
	\]
In particular, if $B$ is selfadjoint boundary condition then $D_{\maxd,B}$ is a selfadjoint operator.

\begin{defi}
	Let $B\subseteq H^{\frac 12}(\partial \Sigma,E|_{\partial\Sigma})$ be a  linear subspace such that $B\subseteq \check{H}(A)$ is closed and  $B^\ast\subseteq H^{\frac{1}{2}}(\partial\Sigma,E|_{\partial\Sigma})$, then $B$ is called \emph{elliptic}.
\end{defi}

\begin{satz}[Theorem~7.11 \cite{BB1}] \label{thm:ell}
	Let $B\subseteq H^{\frac 12}(\partial\Sigma,E|_{\partial\Sigma})$ be a linear subspace.
	Then the following are equivalent:
	\begin{enumerate}[label=(\arabic*), labelwidth=!, labelindent=2pt, leftmargin=21pt]
		\item $\operatorname{dom}(D_{\maxd,B})\subseteq H^1(\Sigma,E)$ and $\operatorname{dom}(D_{\maxd,B^\ast})\subseteq H^1(\Sigma,E)$;
		\item $B$ is elliptic.
	\end{enumerate}
	Moreover, for any elliptic boundary condition $B$ the adjoint boundary condition $B^\ast$ is elliptic as well.
\end{satz}

As a direct consequence we find:
\begin{koro}\label{koro:closed}
	Let $B$ be an elliptic boundary condition.
	Then $\operatorname{dom}(D_{\maxd,B})$ is a closed subspace of  $ H^1(\Sigma,E)$.
	 Moreover, $\norm{\cdot}_D$ and $\norm{\cdot}_{H^1}$ are equivalent on $\operatorname{dom}(D_{\maxd,B})$.
\end{koro}
\begin{proof}
	We first show that $\operatorname{dom}(D_{\maxd,B})$ is closed in $H^1(\Sigma,E)$.
	Let
	\[\{\psi_n\}_{n\in\mathbb{N}}\subseteq \operatorname{dom}(D_{\maxd,B})\subseteq H^1(\Sigma,E)\subseteq \operatorname{dom}(D_{\maxd})\]
	such that $\psi_n\to\psi$ in $H^1(\Sigma,E)$.
	 Since $\norm{\cdot}_D\leq C \norm{\cdot}_{H^1}$ and $(\operatorname{dom}(D_{\maxd,B}),\norm{\cdot}_D)$ is a Banach space, we know that $\psi_n\to \psi\in\operatorname{dom}(D_{\maxd,B})$  and, thus, $\operatorname{dom}(D_{\maxd,B})$ is closed in $H^1(\Sigma,E)$.

	In particular, $(\operatorname{dom}(D_{\maxd,B}),\norm{\cdot}_D)$ and $(\operatorname{dom}(D_{\maxd,B}), \norm{\cdot}_{H^1})$ are both Banach spaces with $\norm{\cdot}_D\leq C \norm{\cdot}_{H^1}$.
	 Hence, the identity map  $\id\colon (\operatorname{dom}(D_{\maxd,B}),\norm{\cdot}_{H^1})\to (\operatorname{dom}(D_{\maxd,B}),\norm{\cdot}_D)$ is continuous and bijective.
	 By the open mapping theorem this is an isomorphism, which implies the second claim of the corollary.
\end{proof}

\begin{bem}
	Theorem~3.12 in \cite{BB2} gives a characterization of selfadjoint elliptic boundary conditions.
	Furthermore, this theorem shows that in our setting the existence of selfadjoint elliptic boundary conditions require the boundary operator $A$ to have even dimensional kernel.
	Later when we consider selfadjoint elliptic boundary conditions this will be indirectly required.
\end{bem}

For higher boundary regularity there is the notion of $\infty$-regular boundary conditions.
For the technical definition we refer the interested reader to \cite{BB1}.
We need to know that $\infty$-regular boundary conditions are in particular elliptic boundary conditions and the following higher boundary regularity holds:
\begin{satz}[Theorem~7.17 in \cite{BB1}] \label{regular}
	Let $B$ be an $\infty$-regular boundary condition.
	Then
	\[
	D_{B}\psi\in H^k(\Sigma,E)~~ \Leftrightarrow~~ \psi\in H^{k+1}(\Sigma,E).
	\]
	for all $k\in\mathbb{N}$ and $\psi\in \operatorname{dom}(D_{B})$.
\end{satz}

\begin{ass}
From now on, let $B$ be an $\infty$-regular selfadjoint boundary condition.
\end{ass}
For $k\in\mathbb{N}$ the operator $D_B^k$ is again selfadjoint and we put
\begin{align}
H^k_{B}(\Sigma,E)
&:=
\operatorname{dom}(D_B^k)
=
\{\psi\in H^k(\Sigma,E);~R(D^l\psi)\in B~~\text{for all } 0\leq l\leq k-1\}.
\label{eq:DefHkB}
\end{align}
The scalar product
\[
\spann{\psi,\phi}_{B,k}:=\spann{\psi,\phi}_{L^2(\Sigma,E)} + \spann{D_B^k\psi,D_B^k\phi}_{L^2(\Sigma,E)}
\]
induces the graph norm of $D_B^k$.
Thus $H^k_{B}(\Sigma,E)$ is a Hilbert space.

\begin{lem}\label{Sobolev:lem}
	For any $k\in\mathbb{N}$ the following holds:
	\begin{enumerate}[label=(\arabic*), labelwidth=!, labelindent=2pt, leftmargin=21pt]
	\item\label{closeda}
	$H^l_B(\Sigma,E)\subseteq H^k_B(\Sigma,E)$ whenever $l>k$.
	\item\label{closedb}
	$H^\infty_B(\Sigma,E):=\bigcap_{l=1}^\infty H^l_{B}(\Sigma,E)$ is a dense subspace of $L^2(\Sigma,E)$.
	\item\label{closedc}
	$H^k_B(\Sigma,E)$ is closed in $H^k(\Sigma,E)$ and the norms $\norm{\cdot}_{B,k}$ and $\norm{\cdot}_{H^k}$ are equivalent on $H^k_B(\Sigma,E)$.
	\item\label{closedd}
	On $H^k_B(\Sigma,E)$, we have
		\begin{align} \label{estimate1}
			\norm{\cdot}_{k,B}^2\leq \norm{\cdot}^2_{B,k-1}+\norm{D\cdot}^2_{B,k-1}.
		\end{align}
	\end{enumerate}
\end{lem}

\begin{proof}
Statements \ref{closeda} and \ref{closedd} are clear from the definitions and \ref{closedb} holds because $H^\infty_B(\Sigma,E)$ contains $C^\infty_{\cc}(\Sigma,E)$.
As to \ref{closedc}, recall that $B$ is closed in $H^{\frac 12}(\partial\Sigma,E|_{\partial\Sigma})$ and observe that
\[
R D^l \colon H^k(\Sigma,E) \to H^{\frac 12}(\partial\Sigma,E|_{\partial\Sigma})
\]
is bounded for $l\leq k-1$.
Thus $H^k_B(\Sigma,E)$ is closed in $H^k(\Sigma,E)$.
Since $D_B^k$ is differential operator of order $k$, we have $\norm{\cdot}_{B,k}\leq C \norm{\cdot}_{H^k}$.
The open mapping theorem applied to the identity mapping on $H^k_B(\Sigma,E)$ implies $\norm{\cdot}_{B,k}\simeq \norm{\cdot}_{H^k}$.
\end{proof}

\begin{bem}
	Note that \ref{closedc} and \ref{closedd} in Lemma~\ref{Sobolev:lem} imply the estimate
	\begin{align}\label{estimate2}
		\norm{\cdot}_{H^k}^2\leq C \left( \norm{\cdot}^2_{H^{k-1}}+\norm{D\cdot}^2_{H^{k-1}}\right)
	\end{align}
	on $H^k_B(\Sigma,E)\subseteq H^{k-1}_B(\Sigma,E)$.
	Moreover, \ref{closedc} in Lemma~\ref{Sobolev:lem} implies that any differential operator $Q:C^\infty(\Sigma,E)\to C^\infty(\Sigma,E)$ of order $l$ extends to a bounded operator
	\[
	Q:H^{k+l}_B(\Sigma,E)\to H^k(\Sigma,E).
	\]
\end{bem}

In Section~\ref{sec:exist} we will need mollifiers to regularize our Cauchy problems.
It will be convient to consider, for $\varepsilon>0$, the operator
\[
\JBeps := \exp\big(-\varepsilon(\id+D_B^2)\big)
\]
defined by functional calculus for selfadjoint operators.
Since the function $x\mapsto e^{-\varepsilon(1+x^2)}$ takes values between $0$ and $e^{-\varepsilon}$, the operator $\JBeps$ is bounded on $L^2(\Sigma,E)$ with operator norm smaller than $1$.

Similarly, the operators $D_B^k \JBeps$ are bounded so that we have a bounded operator $\JBeps:L^2(\Sigma,E)\to H^k_B(\Sigma,E)$ for every $\varepsilon>0$ and $k\in\mathbb{N}_0$.
In particular, $\JBeps\psi\in H^\infty_B(\Sigma,E)$ for each $\psi\in L^2(\Sigma,E)$.
Thus $\JBeps$ is a smoothing operator.

Since $\JBeps$ is a function of $D_B$, these two operators commute.
This and $\norm{\JBeps}_{L^2,L^2}\le e^{-\varepsilon}$ implies
\begin{equation*}
\norm{\JBeps\psi}_{k,B}^2\leq e^{-\varepsilon}\norm{\psi}_{k,B}^2
\end{equation*}
for each $\psi\in H^k_B(\Sigma,E)$.
Thus $\JBeps:H^k_B(\Sigma,E)\to H^k_B(\Sigma,E)$ is a contraction.

Since the family of functions $x\mapsto e^{-\varepsilon(1+x^2)}$ is uniformly bounded and converges pointwise to $1$, the family of operators $\JBeps$ converges strongly to $\id_{L^2(\Sigma,E)}$ as $\varepsilon\searrow 0$.
For $\psi\in H^k_B(\Sigma,E)$ we have $\JBeps\psi\to\psi$ in $L^2(\Sigma,E)$ and $D_B^k\JBeps\psi = \JBeps D_B^k\psi \to D_B^k\psi$ in $L^2(\Sigma,E)$.
Thus $\JBeps\psi\to\psi$ in $H^k_B(\Sigma,E)$, i.e.\ the family of operators $\JBeps$ converges strongly to $\id_{H^k_{B}(\Sigma,E)}$ in the space of bounded operators on $H^k_{B}(\Sigma,E)$.

\medskip
Showing that a boundary condition is elliptic or $\infty$-regular can be difficult but for some classes of boundary conditions one has convenient criteria at hand.
For later use in the Lorentzian setting, the following classes of boundary conditions will be quite accessible:

\begin{defi}[(pseudo-) local boundary condition]\leavevmode
	\begin{enumerate}[label=(\arabic*), labelwidth=!, labelindent=2pt, leftmargin=21pt]
		\item We say that a linear subspace $B\subseteq H^{\frac{1}{2}}(\partial\Sigma,E|_{\partial\Sigma})$ is a \emph{local boundary condition} if there is a subbundle $E'\subseteq E|_{\partial\Sigma}$ such that $B=H^{\frac 12}(\partial\Sigma,E')$.
		\item We say that a linear subspace $B\subseteq H^{\frac 12}(\partial\Sigma,E|_{\partial\Sigma})$ is a \emph{pseudolocal boundary condition} if there is a classical pseudodifferential operator $P$ of order zero, acting on sections of $E$ over $\partial\Sigma$, which induces a projection on $L^2(\partial \Sigma,E|_{\partial\Sigma})$ such that $B=P(H^{\frac 12}(\partial \Sigma,E|_{\partial\Sigma}))$.
	\end{enumerate}
\end{defi}
Since $P$ is a pseudodifferential operator of order zero, it yields a bounded map $H^s(\partial\Sigma,E|_{\partial\Sigma})\to H^s(\partial\Sigma,E|_{\partial\Sigma})$ for every $s\in\mathbb{R}$.
Note that such a $B$ need not be a boundary condition as in Definition~\ref{def:bound} since it may not be closed in the check space.
Here the following will be convenient:
\begin{satz}[Theorem~7.20, Corollary~7.23 \& Proposition~7.24 \cite{BB1}] \label{critEll} \leavevmode
	\begin{enumerate}[label=(\arabic*), labelwidth=!, labelindent=2pt, leftmargin=21pt]
		\item 
		The following are equivalent:
		\begin{enumerate}
			\item $B=P(H^{\frac{1}{2}}(\partial\Sigma,E|_{\partial\Sigma}))$ is closed in $\check{H}(A)$ and elliptic,
			\item $P-\chi^+(A)\colon L^2(\partial\Sigma,E|_{\partial\Sigma})\to L^2(\partial\Sigma,E|_{\partial\Sigma})$ is an elliptic operator, and
			\item $P-\chi^+(A)\colon L^2(\partial\Sigma,E|_{\partial\Sigma})\to L^2(\partial\Sigma,E|_{\partial\Sigma})$ is a Fredholm operator.
		\end{enumerate}
		\item\label{critEll2} 
		Let $E|_{\partial \Sigma}:=E'\oplus E''$ be a decomposition such that the principal symbol $\sigma_A(\zeta)$ of the boundary operator interchanges $E'$ and $E''$ for every $\zeta\in T^\ast\partial \Sigma$.

		Then $B':=H^{\frac 12}(\partial \Sigma,E')$ and $B'':=H^{\frac 12}(\partial \Sigma,E'')$ are closed in $\check{H}(A)$ and elliptic.
		\item 
		Every pseudolocal elliptic boundary condition is $\infty$-regular.
	\end{enumerate}
\end{satz}

We can use this analysis a special class of pseudolocal boundary conditions, the Grassmannian projections, defined as follows: 

\begin{defi} \label{Grassmannian:Def}
A classical pseudodifferential projection $P:L^2(\partial\Sigma,E|_{\partial\Sigma})\to L^2(\partial\Sigma,E|_{\partial\Sigma})$ is called a \emph{Grassmannian projection} if
\begin{enumerate}[label=(\alph*), labelwidth=!, labelindent=2pt, leftmargin=21pt]
	\item\label{Grassmannian1} $P^\ast=P$,
	\item\label{Grassmannian2} $P=\id+\sigma_D(\ngm)P\sigma_D(\ngm)$, and
	\item\label{Grassmannian3} $P-\chi^+(A):L^2(\partial\Sigma,E|_{\partial\Sigma})\to L^2(\partial\Sigma,E|_{\partial\Sigma})$ is a Fredholm operator.
\end{enumerate}
\end{defi}
First, let us analyze the first two properties of Definition~\ref{Grassmannian:Def}.
\begin{koro}\label{Grassmannian:cor1}
Let $P:L^2(\partial\Sigma,E|_{\partial\Sigma})\to L^2(\partial\Sigma,E|_{\partial\Sigma})$ be a classical pseudodifferential projection such that it satisfies Property~\ref{Grassmannian1} and Property~\ref{Grassmannian2} of Definition~\ref{Grassmannian:Def}.  Then $B:=P(H^{\frac 12}(\partial\Sigma,E|_{\partial\Sigma}))$ is a selfadjoint boundary condition for $D$.
\end{koro}

\begin{proof}
	Using Assumptions~\ref{Grassmannian1} and \ref{Grassmannian2} we compute
	\begin{align*}
		\int_{\partial\Sigma}h_E(\sigma_D(\ngm)P\psi,\phi)\dt{\mu}_{\partial\Sigma}
		&=
		-\int_{\partial\Sigma}h_E(\sigma_D(\ngm)P\sigma_D(\ngm)^{2}\psi,\phi)\dt{\mu}_{\partial\Sigma} \\
		&=
		-\int_{\partial\Sigma}h_E((P-\id)\sigma_D(\ngm)\psi,\phi)\dt{\mu}_{\partial\Sigma} \\
		&=
		\int_{\partial\Sigma}h_E(\sigma_D(\ngm)\psi,(\id-P)\phi)\dt{\mu}_{\partial\Sigma} .
	\end{align*}
	This implies
	\begin{align*}
		B^\ast
		&=
		\Big\{\phi\in H^\frac{1}{2}(\partial\Sigma,E|_{\partial\Sigma});~\int_{\partial\Sigma}h_E(\sigma_D(\ngm)P\psi,\phi)\dt{\mu}_{\partial\Sigma}=0~\forall\,\psi\in H^{\frac 12}(\partial\Sigma,E|_{\partial\Sigma})\Big\} \\
		&=
		\Big\{\phi\in H^\frac{1}{2}(\partial\Sigma,E|_{\partial\Sigma});~\int_{\partial\Sigma}h_E(\sigma_D(\ngm)\psi,(\id-P)\phi)\dt{\mu}_{\partial\Sigma}=0~\forall\,\psi\in H^{\frac 12}(\partial\Sigma,E|_{\partial\Sigma})\Big\} \\
		&=
		P(H^\frac 12(\partial\Sigma,E|_{\partial\Sigma}))\\
		&= B .\qedhere
	\end{align*}
\end{proof}

\medskip
A direct consequence of Theorem~\ref{critEll} and Corollary~\ref{Grassmannian:cor1}  is the following:

\begin{koro}\label{Grassmannian:cor2}
	 Let $P:L^2(\partial\Sigma,E|_{\partial\Sigma})\to L^2(\partial\Sigma,E|_{\partial\Sigma})$ be a Grassmannian projection. Then $B:=P(H^{\frac 12}(\partial\Sigma,E|_{\partial\Sigma}))$ is a selfadjoint boundary condition for $D$ is a $\infty$-regular selfadjoint  boundary condition for $D$.
\end{koro}

\medskip
For a pseudodifferential operator $P$ of order $0$ acting on sections of $E|_{\partial\Sigma}\to\partial\Sigma$ and a fixed $k\in\mathbb{N}$ we consider the operator
\begin{gather*}
	\tilde{P}\colon H^k(\Sigma,E)\to \bigoplus_{l=0}^{k-1}(1-P)H^{k-l-\frac 12}(\partial\Sigma,E|_{\partial \Sigma}), \\
	\tilde{P}(\psi) = ((1-P)R\psi,(1-P)R(D\psi),\dots,(1-P)R(D^{k-1}\psi)).
\end{gather*}

\begin{koro}\label{koro:estimate4}
	Let $P\colon H^{\frac{1}{2}}(\partial\Sigma,E|_{\partial\Sigma})\to H^{\frac 12}(\partial\Sigma,E|_{\partial\Sigma})$ be a selfadjoint pseudo\-differential projection such that $B=P(H^{\frac 12}(\partial\Sigma,E|_{\partial\Sigma}))$ is an elliptic boundary condition.
	Let $k\in\mathbb{N}$.
	Then the operator
	\[
	D\oplus \tilde{P}\colon H^k(\Sigma,E)\to H^{k-1}(\Sigma,E)\oplus \bigoplus_{l=0}^{k-1}(1-P)H^{k-l-\frac 12}(\partial\Sigma,E|_{\partial \Sigma}),
	\]
	has finite dimensional kernel and closed image.
	In particular, there exists a constant $C>0$ such that
	\begin{equation}
	\norm{\psi}_{H^k}^2\leq C\Big(\norm{\psi}^2_{H^{k-1}}+\norm{D\psi}^{2}_{H^{k-1}}+\sum_{l=0}^{k-1}\norm{(1-P)R(D^l\psi)}^2_{H^{k-l-\frac 12}}\Big)
	\label{eq:DPest}
	\end{equation}
	for all $\psi\in H^k(\Sigma,E)$.
\end{koro}

\begin{proof}
	By the construction of $\tilde{P}$, we have $\operatorname{ker}\tilde{P}=H^k_B(\Sigma,E)$.
	The embedding $H^{k}_B(\Sigma,E)\hookrightarrow H^{k-1}(\Sigma,E)$ is compact by the Rellich lemma and \eqref{estimate2} says
	\[
	\norm{\psi}_{H^k}^2\leq C \left( \norm{\psi}^2_{H^{k-1}}+\norm{D\psi}^2_{H^{k-1}}\right)
	\]
	for all $\psi\in H^{k}_B(\Sigma,E)$.
	The implication ``(ii) $\Rightarrow$ (i)'' in Proposition~A.3 in \cite{BB1} now shows that $D|_{\operatorname{ker}\tilde{P}}: \operatorname{ker}\tilde{P}\to H^{k-1}(\Sigma,E)$ has finite dimensional kernel and closed image.
	By Proposition~A.1 in \cite{BB1}, $D\oplus \tilde{P}$ has finite dimensional kernel and closed image as well.
	Inequality \eqref{eq:DPest} follows upon applying (iv) of Proposition~A.3 in \cite{BB1}.
\end{proof}
We conclude this section with some examples for elliptic and $\infty$-regular boundary conditions that we will use later on.
\begin{bsp}
		Let $E|_{\partial \Sigma}=E^+\oplus E^-$ be an orthogonal splitting of the bundle $E$ restricted to the boundary.
		Let $\chi$ be the selfadjoint involution of $E|_{\partial \Sigma}$ which acts by $\pm 1$ on $E^\pm$.
		One calls $\chi$ a \emph{boundary chirality} (w.r.t.\ $A$) if $\chi$ anticommutes with $A$.
		The associated boundary conditions $B_{\pm}=H^{\frac 12}(\partial\Sigma, E^\pm)$ are elliptic by Theorem~\ref{critEll}~\ref{critEll2}.

		Moreover, the adjoint boundary conditions are given by $B_{\pm}^\ast=\sigma_D(\ngm)B_{\mp}$. 
		Hence if we additionally assume that $\chi$ anticommutes with $\sigma_D(\ngm)$, the boundary conditions $B_{\pm}$ are selfadjoint.
\end{bsp}
\begin{bsp}\label{bsp_elliptic}
		Recall that $\chi^-(A)$ is the projection onto the sum of eigenspaces of $A$ to nonpositive eigenvalues.
		The \emph{Atiyah-Patodi-Singer} (APS) condition $B_{\APS}=\chi^-(A)H^{\frac{1}{2}}(\partial\Sigma,S\Sigma|_{\partial\Sigma})$ is probably the most prominent example of a non-local boundary condition.
		The APS boundary conditions are pseudo local, see for example Proposition~14.2 in \cite{BW}.
		By Theorem~\ref{critEll}, the APS boundary condition is elliptic and even $\infty$-regular.
		Furthermore, the APS condition is selfadjoint, i.e.\ $B_{\APS}^\ast=B_{\APS}$, if and only if $\operatorname{ker}A=\{0\}$.
\end{bsp}
\begin{bsp} \label{bsp:trans}
		Let $\Sigma$ be a closed Riemannian manifold and $E\to \Sigma$ be a Hermitian vector bundle and $D$ be a Dirac-type operator acting on sections of $E$.
		Let $\mathcal{M}\subseteq \Sigma$ be a compact hypersurface with trivial normal bundle.
		Cut $\Sigma$ along $\mathcal{M}$ to obtain a compact Riemannian manifold $\Sigma'$ with boundary $\partial \Sigma'=\mathcal{M}_1\sqcup \mathcal{M}_2$, where $\mathcal{M}_1$ and $\mathcal{M}_2$ are two copies of $\mathcal{M}$ with opposite relative orientations in $\Sigma'$.
		We get an induced vector bundle $E'\to \Sigma'$ and a Dirac-type operator $D'$ acting on sections of $E'$.

		Any $\psi\in H^1(\Sigma,E)$ yields $\psi'\in H^1(\Sigma',E')$ such that $\psi'|_{\mathcal{M}_1}=\psi'|_{\mathcal{M}_2}$ and $\psi=\psi'$ on $\Sigma\setminus \mathcal{M} = \Sigma'\setminus\partial\Sigma'$.
		This motivates the \emph{transmission conditions} for $D'$ on $\Sigma'$.
		We set
		\[
		B:=\{(\psi,\psi)\in H^{\frac 12}(\mathcal{M}_1,E|_{\mathcal{M}_1})\oplus H^{\frac 12}(\mathcal{M}_2,E|_{\mathcal{M}_2}); \psi \in H^{\frac 12}(\mathcal{M},E|_\mathcal{M})\},
		\]
		where we identify
		\[
		H^{\frac 12}(\mathcal{M}_1,E|_{\mathcal{M}_1})=H^{\frac 12}(\mathcal{M}_2,E|_{\mathcal{M}_2})=H^{\frac 12}(\mathcal{M},E|_{\mathcal{M}}).
		\]
		Let  $A=A_0\oplus -A_0$ be a boundary operator for $D'$, where $A_0$ is a selfadjoint Dirac-type operator on $C^\infty(\mathcal{M},E|_\mathcal{M})$.
		The transmission conditions are $\infty$-regular selfadjoint boundary conditions, see \cite{BB2}.
		Note that the transmission conditions are not pseudolocal.
\end{bsp}

\section{Families of boundary value problems} \label{ConFunCal}

The results of the previous section will be applied to spacelike Cauchy hypersurfaces in a Lorentzian spacetime. 
Due to the additional time dimension, we will have to consider families of boundary value problems, parametrized by a suitable time function.
In this section we will study the continuity properties of such families.

Let $(\Sigma,g)$ be a compact Riemannian manifold with boundary, $E\to \Sigma$ a Hermitian vector bundle and $D_t\colon C^\infty(\Sigma,E)\to C^\infty(\Sigma,E)$ a family of formally selfadjoint Dirac-type operators with coefficients depending smoothly on $t\in\mathbb{R}$.
Let $\ngm$ be the interior unit conormal to $\partial\Sigma$. 
Furthermore, let $\{P_t\}_{t\in\mathbb{R}}$ be a family of orthogonal Grassmannian projections (see Definition~\ref{Grassmannian:Def}) and $\{B_t\}_{t\in\mathbb{R}}$ be the family of the corresponding selfadjoint boundary conditions $B_t=P_t(H^\frac 12 (\partial\Sigma,E|_{\partial\Sigma}))$.
Recall that for each $t\in\mathbb{R}$, $l\in\mathbb{N}$ and $\varepsilon>0$ we have a bounded operator
$$
D_{t,B_t}^lJ_{B_t}^{(\varepsilon)}\colon L^2(\Sigma,E)\to L^2(\Sigma,E).
$$

\begin{lem}\label{HkContinuity}
	Let $\{D_t\}_{t\in\mathbb{R}}$ be as described above, $\{P_t\}_{t\in\mathbb{R}}$ be a family of Grassmannian projections. 
	Assume that $t\mapsto P_t$ is $H^s$-norm continuous for every $s\in\mathbb{N}_0$.
	Then 
	\[
	t\mapsto D_{t,B_t}^lJ_{B_t}^{(\varepsilon)}
	\]
	is $H^k$-norm continuous for every $k,l\in\mathbb{N}_0$ and $\varepsilon>0$.
\end{lem}

\begin{proof}
	The proof is by induction on $k$.
	First, we show the case $k=0$.
		Since $t\mapsto P_t$ is $L^2$-norm continuous, Theorem~3.9 in \cite{Lesch1} implies that $D_{t,B_t}$ is continuous w.r.t.\ the gap metric.
	By Theorem~1.1 in \cite{Lesch1}, the path of the corresponding Cayley transforms $t\mapsto U_t:=\kappa(D_{t,B_t})$ is $L^2$-norm continuous.
	Here $\kappa\colon\mathbb{R}\to S^1\subset\mathbb{C}$ is the function given by $x\mapsto \frac{x-i}{x+i}$.

    The inverse of the Cayley transform is given by the function $S^1\setminus\{1\}\to \mathbb{R}$, $z\mapsto i\frac{1+z}{1-z}$.
    The function $\lambda_\varepsilon\colon S^1\setminus\{1\} \to \mathbb{R}$, $z\mapsto \left(i\frac{1+z}{1-z}\right)^l\cdot e^{-\varepsilon\left(1+i \left(\frac{1+z}{1-z}\right)^2\right)}$ extends continuously to $1$ by zero.
    Proposition~3.2.10. in \cite{CStar} now implies that
    \[
    t \mapsto D_{t,B_t}^lJ_{B_t}^{(\varepsilon)} = \lambda_\varepsilon(U_t)
    \]
    is $L^2$-norm continuous.    
     Hence, the claim is shown for $k=0$. Now we do the induction step $k\mapsto k+1$. 
    
    	Here and in the following, we denote for two Hilbert spaces $\mathscr{H}_1$ and $\mathscr{H}_2$ the operator norm on the space of bounded operators $\mathscr{H}_1\to\mathscr{H}_2$  by $\norm{\cdot}_{\mathscr{H}_1,\mathscr{H}_2}$.
    	
     We show continuity at $t_0\in\mathbb{R}$.
     For $t$ near $t_0$ we we use \eqref{eq:DPest} to estimate:
     	
     \begin{align*}
     	\Big\|D_{t_0,B_{t_0}}^l J_{B_{t_0}}^{(\varepsilon)}&-D_{t,B_{t}}^lJ^{(\varepsilon)}_{B_{t}}\Big\|_{H^{k+1},H^{k+1}}
     	\leq
     	C_1\Big{(}\underbrace{\norm{D_{t_0,B_{t_0}}^l J_{B_{t_0}}^{(\varepsilon)}-D_{t,B_{t}}^lJ^{(\varepsilon)}_{B_{t}}}_{H^{k+1},H^k}}_{=:(\mathrm{I})} \\
     	&+\underbrace{\norm{D_{t_0}\left(D_{t_0,B_{t_0}}^l J_{B_{t_0}}^{(\varepsilon)}-D_{t,B_{t}}^lJ^{(\varepsilon)}_{B_{t}}\right)}_{H^{k+1},H^{k}}}_{=:(\mathrm{II})}  \\
     	&+\underbrace{\sum_{m=0}^{k}\norm{(1-P_{t_0}) RD_{t_0}^m\left(D_{t_0,B_{t_0}}^l J_{B_{t_0}}^{(\varepsilon)}-D_{t,B_{t}}^lJ^{(\varepsilon)}_{B_{t}}\right)}_{H^{k+1},H^{k-m+\frac{1}{2}}}}_{=:(\mathrm{III})}\Big{)} .
     \end{align*}
We estimate the terms (I), (II), and (III) separately.
Using $\norm{\cdot}_{H^k}\leq C \norm{\cdot}_{H^{k+1}}$, we can estimate term (I) by
 \[
 (\mathrm{I})\leq C_2 \norm{D_{t_0,B_{t_0}}^l J_{B_{t_0}}^{(\varepsilon)}-D_{t,B_{t}}^lJ^{(\varepsilon)}_{B_{t}}}_{H^k,H^k},
 \]
 where the RHS goes to zero as $t\to t_0$ by induction hypothesis.
 	For the second term we find
 \begin{align*}
 	(\mathrm{II})
 	&\leq
 	\norm{D^{l+1}_ {t_0,B_{t_0}}J^{(\varepsilon)}_{B_{t_0}}-D^{l+1}_{t,B_{t}}J^{(\varepsilon)}_{B_{t}}}_{H^{k+1},H^k}
 	+\norm{(D_{t}-D_{t_0})D^{l}_{t,B_{t}}J^{(\varepsilon)}_{B_{t}}}_{H^{k+1},H^k}  \\
 	&\leq C_3 \norm{D^{l+1}_ {t_0,B_{t_0}}J^{(\varepsilon)}_{B_{t_0}}-D^{l+1}_{t,B_{t}}J^{(\varepsilon)}_{B_{t}}}_{H^{k},H^k}
 	+\norm{(D_{t}-D_{t_0})D^{l}_{t,B_{t}}J^{(\varepsilon)}_{B_{t}}}_{H^{k+1},H^k} .
 \end{align*}
The first summand again tends to zero as $t\to t_0$ by induction hypothesis.
As for the second, we see
\begin{align*}
	\norm{(D_{t}-D_{t_0})D^{l}_{t,B_{t}}J^{(\varepsilon)}_{B_{t}}}_{H^{k+1},H^k}
	\leq
	\underbrace{\norm{D_{t}-D_{t_0}}_{H^{k+1},H^k}}_{=:(\mathrm{IIa})}\underbrace{\norm{D^{l}_{t,B_{t}}J^{(\varepsilon)}_{B_{t}}}_{H^{k+1},H^{k+1}}}_{=:(\mathrm{IIb})} .
\end{align*}
Since the coefficients of $D_{t}$ depend smoothly on $t$, the term (IIa) tends to zero as $t\to t_0$.
Furthermore, the range of $D^{l}_{t,B_{t}}J^{(\varepsilon)}_{B_{t}}$ is contained in $H^\infty_{B_{t}}(\Sigma,E)$.
Thus applying \eqref{estimate1} iteratively yields
\[
(\mathrm{IIb}) \le
C_4 \sum_{m=l}^{l+k+1} \norm{D^{m}_{t,B_{t}}J^{(\varepsilon)}_{B_{t}}}^2_{L^2,L^2}.
\]
	The RHS is a continuous function by the induction claim for $k=0$.
		Hence (IIb) remains bounded for $t$ near $t_0$ and term (II) goes to zero.
		
		It remains to control term (III).
		Each summand in (III) can be estimated as follows:
			\begin{align}
			\Big\|&(1-P_{t_0}) RD_{t_0}^m\left(D_{{t_0},B_{t_0}}^l J_{B_{t_0}}^{(\varepsilon)}-D_{t,B_{t}}^lJ^{(\varepsilon)}_{B_{t}}\right)\Big\|_{H^{k+1},H^{k-m+\frac{1}{2}}} \notag\\
			&=
			\Big\|(1-P_{t_0}) RD_{t_0}^mD_{t,B_{t}}^lJ^{(\varepsilon)}_{B_{t}}\Big\|_{H^{k+1},H^{k-m+\frac{1}{2}}} \label{eq:III1}\\
			&\leq 
			\norm{(1-P_{t_0}) RD_{t,B_{t}}^{l+m}J^{(\varepsilon)}_{B_{t}}\psi}_{H^{k+1},H^{k-m+\frac{1}{2}}}+\norm{(1-P_{t_0}) R(D_{t}^m-D_{t_0}^m)D_{t,B_{t}}^{l}J^{(\varepsilon)}_{B_{t}}}_{H^{k+1},H^{k-m+\frac{1}{2}}} \notag \\
			&=
			\norm{(P_{t}-P_{t_0}) RD_{t,B_{t}}^{l+m}J^{(\varepsilon)}_{B_{t}}}_{H^{k+1},H^{k-m+\frac{1}{2}}}
			+\norm{(1-P_{t_0}) R(D_{t}^m-D_{t_0}^m)D_{t,B_{t}}^{l}J^{(\varepsilon)}_{B_{t}}}_{H^{k+1},H^{k-m+\frac{1}{2}}} \label{eq:III2} \\
				&\leq  
			\norm{(P_{t}-P_{t_0}) RD_{t,B_{t}}^{l+m}J^{(\varepsilon)}_{B_{t}}}_{H^{k+1},H^{s_0}}
			+\norm{(1-P_{t_0}) R(D_{t}^m-D_{t_0}^m)D_{t,B_{t}}^{l}J^{(\varepsilon)}_{B_{t}}}_{H^{k+1},H^{s_0}}  \label{eq:III3} \\
			&\leq 
			C_5\norm{P_{t}-P_{t_0}}_{H^{s_0},H^{s_0}}\norm{D_{t,B_{t}}^{l+m}J^{(\varepsilon)}_{B_{t}}}_{H^{k+1},H^{s_0+\frac 12}}\notag\\
			&\quad +C_5 \norm{1-P_{t_0}}_{H^{s_0},H^{s_0}}\norm{(D_{t}^m-D_{t_0}^m)D_{t,B_{t}}^{l}J^{(\varepsilon)}_{B_{t}}}_{H^{k+1},H^{s_0+\frac 12}} \label{eq:III4}\\
			&\leq  
			C_5\underbrace{\norm{P_{t}-P_{t_0}}_{H^{s_0},H^{s_0}}}_{=:\mathrm{(IIIa)}} \underbrace{\norm{D_{t,B_{t}}^{l+m}J^{(\varepsilon)}_{B_{t}}}_{H^{k+1},H^{s_0+\frac 12}}}_{=:\mathrm{(IIIb)}}\notag\\
			&\quad+C_5 \underbrace{\norm{1-P_{t_0}}_{H^{s_0},H^{s_0}}}_{=:\mathrm{(IIIc)}} \underbrace{\norm{D_{t}^m-D_{t_0}^m}_{H^{s_0+m+\frac 12},H^{s_0+\frac 12}}}_{=:\mathrm{(IIId)}} \underbrace{\norm{D_{t,B_{t}}^{l}J^{(\varepsilon)}_{B_{t}}}_{H^{k+1},H^{s_0+m+\frac 12}}}_{=:\mathrm{(IIIe)}} . \notag
		\end{align}
		In \eqref{eq:III1} we used that the range of $J_{B_{t_0}}^{(\varepsilon)}$ is contained in $H^\infty_{B_{t_0}}(\Sigma,E)$ so that $RD_{t_0}^{m+l} J^{(\varepsilon)}_{B_{t_0}}\psi$ lies in the kernel of $1-P_{t_0}$.
	A similar argument yields \eqref{eq:III2}.
	In \eqref{eq:III3} we have chosen $s_0$ large enough so that $s_0\ge k-m+\frac12$.
	The trace theorem yields \eqref{eq:III4}.
		
	Now observe that the expressions (IIIa) and (IIId) tend to zero as $t\to t_0$ while (IIIb), (IIIc), and (IIIe) remain bounded.
	Thus (III) tends to zero which concludes the proof.
\end{proof}

The following special case of Lemma~\ref{HkContinuity} will be convenient later for discussing explicit examples in Section~\ref{sec:examples}.

\begin{lem} \label{APS-continuity}
	Let $\{D_t\}_{t\in\mathbb{R}}$ be as described above.
	Let $\{A_t\}_{t\in\mathbb{R}}$ be a family of corresponding boundary operators on $\partial\Sigma$ such that the coefficients of $A_t$ depend smoothly on $t$. 
	Let $\{P_t\}_{t\in\mathbb{R}}$ be a family of Grassmannian projections such that $P_t A_t=A_t P_t$ for all $t$ or $P_t A_t=-A_t P_t$ for all $t$.
	Assume that $t\mapsto P_t$ is $L^2$-norm continuous. 
	Then
	\[
	t\mapsto D_{t,B_t}^l J_{B_t}^{(\varepsilon)}
	\]
	is $H^k$-norm continuous for all $k,l\in\mathbb{N}_0$ and $\varepsilon>0$.	
\end{lem}
\begin{proof}
	In order to apply Lemma~\ref{HkContinuity} we will show by induction that $t\mapsto P_t$ is $H^k$-norm continuous for all $k\in\mathbb{N}_0$.
	For $k=0$ this holds by assumption.
	To perform the induction step, we assume $H^k$-norm continuity and show $H^{k+1}$-norm continuity.
	
	We show continuity at $t_0\in\mathbb{R}$.
	For $t$ near $t_0$ we estimate:
	\begin{align}
		\norm{P_{t_0}-P_{t}}_{H^{k+1},H^{k+1}}
		&\leq
		C\big(\norm{A_{t}(P_{t_0}-P_{t})}_{H^{k+1},H^{k}}+\norm{P_{t_0}-P_{t}}_{H^{k+1},H^{k}}\big) \label{eq:APScont1}\\
		&\leq
		C\big(\norm{A_{t_0}P_{t_0}-A_{t}P_{t}}_{H^{k+1},H^{k}} + \norm{(A_{t_0}-A_{t})P_{t_0}}_{H^{k+1},H^{k}} \notag\\
		&\quad+ \norm{P_{t_0}-P_{t}}_{H^{k+1},H^{k}}\big) \notag\\
		&\leq
		C\big(\norm{P_{t_0}A_{t_0}-P_{t}A_{t}}_{H^{k+1},H^{k}} \notag\\
		&\quad+\norm{A_{t}-A_{t_0}}_{H^{k+1},H^{k}} \norm{P_{t_0}}_{H^{k+1}, H^{k+1}}+\norm{P_{t_0}-P_{t}}_{H^{k},H^{k}}\big). \label{eq:APScont2}
	\end{align}
	In \eqref{eq:APScont1} we used the elliptic estimate for the first-order operators $A_{t}$ on the closed manifold $\partial\Sigma$.
	Since the coefficients of $A_{t}$ depend smoothly on $t$, the constant $C$ can be chosen uniformly for $t$ near $t_0$.
	In \eqref{eq:APScont2} we used that $P_{t}A_t=\pm A_tP_t$.

	As $t\to t_0$, the last term in \eqref{eq:APScont2} tends to zero by induction hypothesis.
	The second term converges to zero because $A_t$ is a bounded operator from $H^{k+1}$ to $H^{k}$ and has coefficients depending smoothly on $t$.
	The first one goes to zero by both arguments combined.
	Hence the projection family is $H^{k+1}$-norm continuous and the claim follows.
\end{proof}

\begin{bem}
	In Remark~\ref{bm:closedDirac} we discuss the natural boundary operator for the Riemannian spin Dirac operator.
	This choice of boundary operator depends on the spin Dirac operator itself and on the mean curvature of the boundary.
	Hence, in this setting, $D_t$ having coefficients depending smoothly on $t$ implies directly that the coefficients of $A_t$ also depend smoothly on $t$.
\end{bem}

\begin{bem}\label{bem-APS-cont}
	To apply Lemma~\ref{APS-continuity} to the APS boundary conditions let us additionally assume that $\operatorname{ker}A_t=\{0\}$ for all $t\in\mathbb{R}$.
	Since $A_t$ has coefficients depending smoothly on $t$ and $\partial\Sigma$ is a closed manifold, Lemma~3.3 in \cite{Lesch2} shows that $\chi^-(A_t)$ is $L^2$-norm continuous.
	Furthermore, the projection $\chi^-(A_t)$ commutes with $A_t$.
	Thus Lemma~\ref{APS-continuity} applies.
\end{bem}

\section{The Dirac operator on Lorentzian manifolds}
\label{sec:preliminaries}

Next we collect the most important facts about spacetimes with timelike boundary and the Lorentzian Dirac operator. 
We will require some familiarity with Loren\-tzian geometry and spin geometry. 
The reader may consult \cite{Oneill}  for an introduction to Lorentzian geometry and \cite{BGM,Baum} for semi-Riemannian spin geometry.

\subsection{Spacetimes with timelike boundary} \label{sec:timelike} 
Let us start by recalling the concept of timelike boundary.

\begin{defi} 
	A \emph{Lorentzian manifold with timelike boundary} $(M,g)$ is a Lorentzian manifold with boundary such that $\iota^\ast g$ is a Lorentzian metric on the boundary where $\iota\colon\partial M\hookrightarrow M$ is the natural inclusion map.
	
	A \emph{spacetime with timelike boundary} is a time-oriented Lorentzian manifold with timelike boundary. 
\end{defi}
A time orientation on $M$ can be represented by a continuous timelike vector field on $M$.
There is no loss of generality in assuming that this vector field is tangent to the boundary along $\partial M$, compare Proposition~2.4 in \cite{AFS}.
Hence the boundary inherits a time orientation.

\begin{bsp}\label{bsp_timelike}
	\begin{enumerate}[label=(\arabic*), wide, labelwidth=!, labelindent=0pt]
		\item \label{bsp_timelike1}
		Let $M=\mathbb{R}\times\bar{B}_1(0)$ where $\bar{B}_1(0)\subseteq\mathbb{R^2}$ is the 2-dimensional closed unit disk.
		The embedding of $M$ into the Minkowski space $(\mathbb{R}^{1,2},g_\mathrm{Min})$ provides $M$ with a Lorentzian metric turning $(M,g_\mathrm{Min})$ into a spacetime with timelike boundary. 
		\item \label{bsp_timelike2}
		The Anti-deSitter (AdS) and the Anti-deSitter-Schwarzschild spacetimes are two famous examples.
		They have conformal timelike boundary, i.\,e.\ after compactifying in radial direction and carefully gluing in the boundary, they become spacetimes with timelike boundary. 
		We refer the interested reader to \cite{Frances,Ammon-Erdmenger,Sowkorsky}. 
	\end{enumerate}
\end{bsp}

As in the case of spacetimes without boundary, we call $(M,g)$ \emph{globally hyperbolic}, if there are no causal loops and the causal diamonds $J^+(p)\cap J^-(q)$ are compact for all $p,q\in M$.
A subset $\Sigma\subseteq M$ is called a \emph{Cauchy hypersurface} if it is intersected exactly once by every inextensible timelike curve. 

By Theorem~1.1 in \cite{AFS}, global hyperbolicity implies that there exists a Cauchy temporal function $T\colon M\to\mathbb{R}$ meaning that the gradient of $T$ is past-directed timelike and the level sets $\Sigma_t = T^{-1}(t)$ are smooth spacelike Cauchy hypersurfaces.
Moreover, $(M,g)$ is isometric to $(\mathbb{R}\times \Sigma, -N^2\dt{t}^2+g_t)$ where $\Sigma$ is a manifold with boundary, $N\in C^\infty(\mathbb{R}\times\Sigma)$ is strictly positive, and $g_t$ is a Riemannian metric on $\{t\}\times\Sigma$.

We will identify $M=\mathbb{R}\times \Sigma$ and $\Sigma_t=\{t\}\times\Sigma$.
The function $N$ will be called the induced \emph{lapse function}.
Let $\nu$ be the past pointing timelike unit vector field perpendicular to $\{\Sigma_t\}_{t\in\mathbb{R}}$.
Clearly, $\nu$ is proportional to the gradient of $T$.
Finally, again by Theorem~1.1 in \cite{AFS}, we may and will assume that $\nu$ is tangential to $\partial M$ along the boundary of $M$.
Then the interior unit normal field $\eta$ along $\partial M$ is perpendicular to $\nu$ and hence tangential to $\Sigma_t$ for every $t$, see Figure~\ref{fig:spacetime}.

\begin{figure}[h]
\begin{overpic}[scale=0.3]{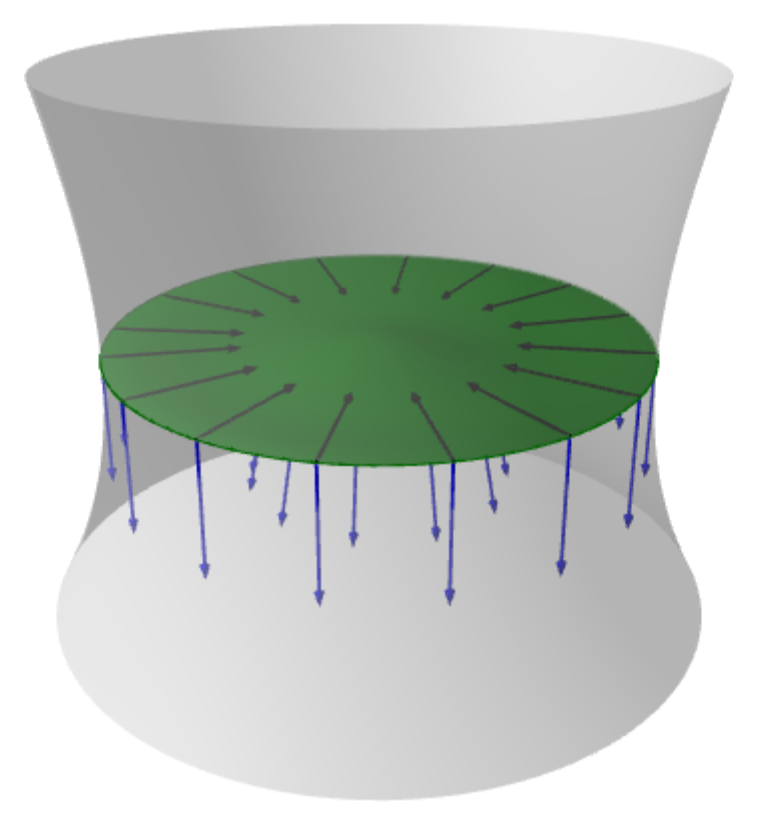}
\put(82,52){$\eta$}
\put(87,70){$\partial M$}
\put(41.5,22){\textcolor{blue}{$\nu$}}
\put(47,72){\textcolor[rgb]{0,.5,0}{$\Sigma_t$}}
\end{overpic}

\caption{Vector fields $\eta$ and $\nu$.}
\label{fig:spacetime}
\end{figure}
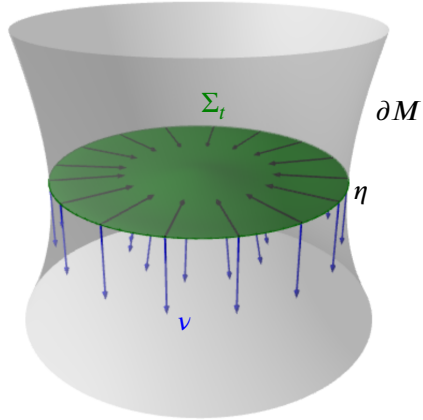

\begin{bsp}
	Example~\ref{bsp_timelike}~\ref{bsp_timelike1} is clearly globally hyperbolic while the discussion for Example~\ref{bsp_timelike}~\ref{bsp_timelike2} is more involved. 
	The boundaryless AdS spacetime has compact time, hence it is not causal and therefore not globally hyperbolic.
	By a suitable conformal change and compactification (see for example \cite{Sowkorsky}), one transforms the AdS spacetime into a spacetime with infinite time that is isometric to $\mathbb{R}\times B_1(0)$ with a cylindrical metric.
	Note that this is still not a globally hyperbolic spacetime, but by gluing in a timelike boundary in a suitable way, one ends up with a globally hyperbolic spacetime with timelike boundary, see for example \cite{Frances}.
\end{bsp}

\begin{bem}\label{bem:embedding}
We mentioned that, by the results in \cite{AFS}, any globally hyperbolic manifold $(M,g)$ with timelike boundary is isometric to $(\mathbb{R}\times \Sigma, -N^2\dt{t}^2+g_t)$.
If $\Sigma$ is compact (with or without boundary), then the converse is also true;
a Lorentzian manifold of the form $(\mathbb{R}\times \Sigma, -N^2\dt{t}^2+g_t)$ is globally hyperbolic.

Now, if $M$ has a timelike boundary, we can double $M$ along $\partial M$, i.e.\ we put $\bar{M}:=M \sqcup_{\partial M} M' = \mathbb{R}\times(\Sigma\sqcup_{\partial\Sigma}\Sigma')$.
Here $M'$ and $\Sigma'$ are copies of $M$ and $\Sigma$, respectively.
We can now extend $N$ to a positive smooth function $\bar{N}$ on $\bar{M}$ and the metrics $g_t$ to smooth Riemannian metrics $\bar{g}_t$ on $\Sigma\sqcup_{\partial\Sigma}\Sigma'$ so that they still depend smoothly on $t$.
Thus we have extended the Lorentzian metric on $M$ to a globally hyperbolic Lorentzian metric $-\bar{N}^2\dt{t}^2+\bar{g}_t$ on $\bar{M}$.
Hence, any spatially compact globally hyperbolic manifold with timelike boundary can be isometrically embedded into a globally hyperbolic manifold without boundary of the same dimension.
\end{bem}

\subsection{The spin Dirac operator on spacetimes with timelike boundary}\label{SubS:Dirac} 
Let $(M,g)$ be a globally hyperbolic \emph{spin} manifold with timelike boundary.
Denote the dimension of $M$ by $n+1$.
For the physically relevant case $n=3$, oriented globally hyperbolic spacetimes are always spinnable.
 
\medskip
\emph{From now onwards, we will always assume $M$ to be spatially compact, i.\,e.\ the Cauchy hypersurfaces of $M$ are compact.}
 
\medskip
Let $SM\to M$ be the complex spin bundle with its invariantly defined non-degenerate inner product $\SM{\cdot}{\cdot}$ and $\nabla^{SM}$ its metric connection. 
Let us denote by $C^\infty(M,SM)$ the space of smooth sections of $SM$.
Furthermore, for tangent vectors $X\in T_xM$ let us denote the \emph{Clifford multiplication} by $\gamma(X)\colon S_xM\to S_xM$.
Clifford multiplication is symmetric with respect  to $\SM{\cdot}{\cdot}$ and satisfies the Clifford relation
\begin{align*}
	\gamma(X)\gamma(Y)+\gamma(Y)\gamma(X)=-2g(X,Y)
\end{align*}
for all $x\in M$ and $X,Y\in T_xM$.
Furthermore, Clifford multiplication is parallel with respect to $\nabla^{SM}$ and the Levi-Civita connection $\nabla$ on $TM$, i.\,e.\ for $X,Y\in C^\infty(M,TM)$ and $\psi\in C^\infty(M,SM)$ one has
\begin{align*}
	\nabla^{SM}_X(\gamma(Y)\psi)=\gamma(Y)\nabla^{SM}_X \psi+\gamma(\nabla_X Y)\psi.
\end{align*}
 \medskip
The \emph{Lorentzian spin Dirac operator} $D$ acts on sections of $SM$.
Locally it is given by
\[
D=\sum_{j=0}^n \varepsilon_j \gamma(e_j)\nabla^ {SM}_{e_j},
\]
where $e_0,e_1,\dots,e_n$ is a Lorentzian orthonormal tangent frame and $\varepsilon_j=g(e_j,e_j)=\pm1$.

Let $\eta$ be the inward pointing spacelike unit normal field to $\partial M$, then the divergence theorem implies for all $\psi,\phi\in C^\infty_\c(M,SM)$, the space of compactly supported smooth spinors,
\begin{align}\label{Green}
	\int_M\SM{D\psi}{\phi}+\SM{\psi}{D\phi}\dt{\mu}_{M}=-\int_{\partial M} \SM{\gamma(\eta)\psi,\phi}\,\dt{\mu}_{\partial M},
\end{align}
where $\dt{\mu}_M$ is the volume element on $M$ with respect to $g$ and $\dt{\mu}_{\partial M}$ is the induced one  on $\partial M$.
\begin{bem}
	Note that \eqref{Green} does not imply that $D$ is formally anti-selfadjoint in a functional analytic sense.
	Since the Green formula above is using an \emph{indefinite} inner product, the integral does not define a $L^2$-scalar product on $M$ giving rise to a Hilbert space. To simplify the following discussion, we will still call an operator \emph{formally anti-selfadjoint} if it satisfies \eqref{Green}. 
\end{bem}

Spinorial data on the spacetime $M$ and on the Cauchy hypersurface $\Sigma_t$ are related as follows.
If $n$ is even, there is a canonical isomorphism
\[
SM|_{\Sigma_t}\cong S\Sigma_t
\]
with $\gamma(X)=-i\gamma(\nu)\gamma_t(X)$ and $\SM{\cdot}{\cdot}=\spann{\gamma(\nu),\cdot,\cdot}_{S\Sigma_t}$. 
Here $\gamma_t$ is the Clifford multiplication on $\Sigma_t$.
If $n$ is odd, then
\[
SM|_{\Sigma_t}\cong S\Sigma_t\oplus S\Sigma_t,
\]
with 
\[\gamma(X)=\Mat{0,i\gamma(\nu)\gamma_t(X);-i\gamma(\nu)\gamma_t(X),0},\]
and 
\[
\SM{\cdot}{\cdot}=\spann{\Mat{0,\gamma(\nu);\gamma(\nu),0}\cdot,\cdot}_{S\Sigma_t\oplus S\Sigma_t}.
\]
In both cases, $n$ being odd or even, the  inner product
\begin{align*}
	\spann{\cdot,\cdot}_0:=
	\spann{\gamma(\nu)\cdot,\cdot}_{SM}.
\end{align*}
is positive definite on $SM|_{\Sigma_t}$ since $\spann{\cdot,\cdot}_{S\Sigma_t}$ is positive definite.
We denote the corresponding norm by $\abs{\cdot}_0:=\sqrt{\spann{\cdot,\cdot}_0}$.
 For each $t\in\mathbb{R}$, let $L^2(\Sigma_t, SM|_{\Sigma_t})$ be the canonical $L^2$-space on the spinor bundle $SM|_{\Sigma_t}$ defined using the volume form $\dt{\mu}_{\Sigma_t}(g_t)$ and the metric $\spann{\cdot,\cdot}_0$. 
 
Using the Gauß formula for $\nabla^{SM}$,
\[
\nabla^{SM}_X\psi=\nabla^{S\Sigma_t}_X\psi-\frac{1}{2}\gamma(\nu)\gamma(W(\cdot))\psi,
\]
where $W$ is the Weingarten map of the Levi-Civita connection on the tangent bundle, we get
\begin{align}\label{splitting1}
	D=-\gamma(\nu)\left[\nabla^{SM}_\nu+iD_t-\frac{n}{2}H_t\right],
\end{align}
where $H_t$ is the mean curvature of $\Sigma_t$ with respect to $\nu$ and
\begin{equation}
D_t=\cancel{D}_t
\quad\text{or}\quad
D_t=\Mat{\cancel{D}_t,0;0,-\cancel{D}_t}
\label{eq:Dt}
\end{equation}
for $n$ even and odd, respectively.
Here $\cancel{D}_t$ is the Riemannian spin Dirac operator on $S\Sigma_t$.
In particular, $D_t$ is a Dirac-type operator whose principal symbol is given by
\[
\sigma_{D_t}(\zeta)=-i\sigma_D(\nu)\sigma_D(\zeta),
\]
for $\zeta\in T_x\Sigma_t$.

\begin{bem}\label{bm:closedDirac}
	The principal symbol $\sigma_{D_t}$ is symmetric with respect to $\spann{\cdot,\cdot}_{SM}$ but  is skew-symmetric with respect to $\spann{\cdot,\cdot}_0$. Furthermore, $D_t$ is formally selfadjoint with respect to the $L^2$-scalar product arising from $\spann{\cdot,\cdot}_0$.

	\medskip
	Since $D_t$ is essentially the Riemannian Dirac operator, one can split $D_t$ along the hypersurface $\partial\Sigma_t\subseteq\Sigma_t$ as
	\begin{align}
		D_t=\sigma_{D_t}(\eta_t^\flat)^{-1}\left(\nabla^{S\Sigma_t}_{\eta_t}+A_t-\frac{n-1}{2}H_t^{\partial\Sigma_t}\right),
		\label{eq:closedDirac}
	\end{align}
	where $\eta_t:=\eta(t,\cdot)$ is the unit normal field $\eta$ to the timelike boundary $\partial M$ restricted to $\Sigma_t$, $H_t^{\partial\Sigma_t}$ is the mean curvature of $\partial\Sigma_t$ inside $\Sigma_t$ with respect to $\eta_t$ and $A_t$ is the double of the Riemannian Dirac operator on $\partial\Sigma_t$. 
	Note, that $\eta_t^\flat$ is taking the role of $\ngm$ in Section~\ref{EllipticBoundary}.
	Since $\partial\Sigma_t$ is a closed manifold, we see that $A_t$ is essentially selfadjoint.
	Thus, the spectrum of $A_t$ is real and discrete.
	Furthermore, all eigenspaces $E(A_t,\lambda)$ of $A_t$ are finite dimensional, consist of smooth sections,  and one has
	\[
	L^2(\partial\Sigma_t,SM|_{\partial\Sigma_t})=\overline{\bigoplus_\lambda E_\lambda(A_t)}^{L^2}.
	\]
	as a Hilbert space sum decomposition.
\end{bem}

\subsection{Conformal change of the metric} \label{conformal}
In order to rewrite the Dirac operator in an analytically manageable form, it will be convenient to consider the conformally equivalent Lorentzian metric $\hat{g}:=N^{-2}g=-\dt{t}^2+N^{-2}g_t$ on $M$. 
For $\hat g$, the $t$-lines are geodesics which is not true in general for the original metric $g$.

For tangent vectors $X$ which are of unit length w.r.t.\ $g$, we write $\hat{X}=N\cdot X$ for the corresponding unit vector w.r.t.\ $\hat{g}$. 
Similarly, we will decorate all other objects corresponding to the new metric with a hat.
Specifically, the spinor bundles for $g$ and $\hat g$ can be canonically identified such that
\begin{enumerate}[label=(\arabic*), labelwidth=!, labelindent=2pt, leftmargin=21pt]
	\item the inner product $\SM{\cdot}{\cdot}$ is unchanged, i.e.\ the identification is an isometry.
	\item $\hat{\gamma}(X)=N^{-1}\gamma(X)$. In particular, $\hat{\gamma}(\hat{\nu})=\gamma(\nu)$ and hence $\spann{\cdot,\cdot}_0$ is also unchanged. 
	\item $\hat{\nabla}^{SM}_X=\nabla^{SM}_X+\frac{N}{2}\left(\gamma(X)\gamma(\nabla N^{-1})-X(N^{-1})\right)$.
	\item $\hat{D}=N^{\frac{n+2}{2}}DN^{-\frac{n}{2}}$.
\end{enumerate}
See \cite{Hijazi} for details.
\begin{bem}
	Recall that $D$ is formally anti-selfadjoint w.r.t.\ the volume measure $ \dt{\mu}_{M}(g)$ for the metric $g$.
	The transformation rule for $\hat D$ is consistent with $\hat D$ being formally anti-selfadjoint w.r.t.\ the volume measure $\dt{\mu}_M(\hat{g})=N^{-n-1}\dt{\mu}_{M}(g)$. 
\end{bem}
In the following we will denote $\hat{ \Sigma}_t$ and $\hat{M}$ for $(\Sigma_t,\hat{ g}_t)$ and $(M,\hat{ g})$, respectively, to emphasize the change of metric on these manifolds.
Using \eqref{splitting1} for $\hat{D}$, we find
\[
\hat{D}=-\hat{\gamma}(\hat{\nu}) \left(\nabla^{S\hat{M}}_{\hat{\nu}}+i\hat{D}_t-\frac{n}{2}\hat{H}_t\right),
\]
where $\hat{D}_t=\hat{\cancel{D}}_t$ for $n$ even and $\hat{D}_t=\Mat{\hat{\cancel{D}}_t,0;0,-\hat{\cancel{D}}_t}$ for $n$ odd. 
Since the Riemannian Dirac operator changes under this conformal change as $\hat{\cancel{D}}_t=N^{\frac{n+1}{2}}\cancel{D}_tN^{-\frac{n-1}{2}}$, we have
\begin{align*}
	\hat{D}_t=N^{\frac{n+1}{2}}D_t N^{-\frac{n-1}{2}}.
\end{align*}
\begin{bem}
	Note that $\hat{D}_t$ is indeed formally selfadjoint and its principal symbol is given as 
	\begin{align}\label{conformalOpPrin}
		\sigma_{\hat{D}_t}(\zeta)=N^{\frac{n+1}{2}}\sigma_{D_t}(\zeta)N^{-\frac{n-1}{2}}=N\sigma_{D_t}(\zeta)
	\end{align}
	for all $x\in\Sigma_t$ and $\zeta\in T^\ast_x\Sigma_t$.
\end{bem}

\subsection{Identifying the Cauchy hypersurfaces}\label{CauchyId}
We will use the foliation of $M$ by Cauchy hypersurfaces to separate spacetime into space and time.
This requires an identification of the Cauchy hypersurfaces which we describe now.

For $t,s\in\mathbb{R}$, let $\tau^s_t\colon S\hat{M}|_{\hat{\Sigma}_t}\to S\hat{M}|_{\hat{\Sigma}_s}$ be the parallel transport (w.r.t.\ $\hat{\nabla}^{SM}$) along the integral curves of $\partial_t$.
By the same symbol we denote the parallel transport $\tau^s_t\colon T\hat{\Sigma}_t\to T\hat{\Sigma}_s$ w.r.t.\ to the Levi-Civita connection of $\hat g$.
We also write $\tau_t:=\tau^0_t$ to simplify the notation.
\begin{bem}
	\begin{enumerate}[label=(\alph*), wide, labelwidth=!, labelindent=0pt]
		\item
		In Subsection~\ref{sec:timelike} we pointed out that it is essential to assume that the Cauchy temporal function has gradient tangential to the boundary to ensure that the parallel transport is globally defined.
		This enables us to identify all Cauchy hypersurfaces with each other.
		\item
		Another implication of $\partial_t$ being tangential to $\partial \hat{M}$ is that the restriction of the parallel transport to the boundary $\tau^s_t|_{\partial \hat{M}}$ also identifies the spinor bundles $S\hat{M}|_{\partial\hat{\Sigma}_t}$ (as well as $T\hat{\Sigma}_t|_{\partial\hat{\Sigma}_t}$) of the  Cauchy hypersurfaces of the globally hyperbolic spacetime $\partial \hat{M}$ with each other.
		\item
		Furthermore, since $\hat{\nu}=\partial_t$ is geodesic, we see that the Clifford multiplication on $S\hat{M}|_{\Sigma_t}$ changes under $\tau_t$ as
		\begin{align*}
		\tau_t(\hat{\gamma}(X)v)
		&=
		\tau_t\big(-i\hat{\gamma}(\hat{\nu})\hat{\gamma}_t(X)v\big)
		=
		-i\hat{\gamma}(\hat{\nu})\tau_t(\hat{\gamma}_t(X)v) \\
		&=
		-i\hat{\gamma}(\hat{\nu})\hat{\gamma}_0(\tau_t X)\tau_t v
		=
		\hat{\gamma}(\tau_t X)\tau_t v
		\end{align*}
		for both $n$ even and odd.
		Furthermore, since $\tau_t$ is an isometry for $\spann{\cdot,\cdot}_{SM}$, we have for $u,v\in S_x\hat{M}$ with $x\in\hat{\Sigma}_t$:
		\[
		\spann{\tau_t u,\tau_t v}_0
		=
		\spann{\hat{\gamma}(\hat{\nu})\tau_t u,\tau_t v}_{SM}
		=
		\spann{\tau_t(\hat{\gamma}(\hat{\nu})u),\tau_t v}_{SM}
		=
		\spann{\hat{\gamma}(\hat{\nu})u,v}_{SM}
		=
		\spann{u,v}_0 .
		\]
		Hence $\tau_t$ is also an isometry for $\spann{\cdot,\cdot}_{0}$.
	\end{enumerate}
\end{bem}

We write $\hat\Sigma = \hat\Sigma_0$ for simplicity and define the \emph{volume distortion function} 
$$
\rho\colon\R\times\hat\Sigma\to\R_+ \quad\text{by}\quad (\tau_0^t)^*\vol_t=\rho(t,\cdot) \cdot \vol_0 .
$$
Here $\vol_t$ denotes the volume form on $\hat\Sigma_t$.
Then the map
\[
U(t):=\rho(t,\cdot)\tau_t\colon L^2(\hat{\Sigma}_t,S\hat{M}|_{\hat{\Sigma}_t})\to L^2(\hat{\Sigma},S\hat{M}|_{\hat{\Sigma}})
\]
is an isometry and $U(0)=\id$.
With the same definition, $U(t)$ also yields an isomorphism $H^s(\hat{\Sigma}_t,S\hat{M}|_{\hat{\Sigma}_t})\to H^s(\hat{\Sigma},S\hat{M}|_{\hat{\Sigma}})$ between the Sobolev spaces for any $s\in\mathbb{R}$.
The family $\{H^s(\hat{\Sigma}_t, S\hat{M}|_{\hat{\Sigma}_t})\}_{t\in\mathbb{R}}$ can be considered as a bundle of Hilbert spaces over $\mathbb{R}$ that is trivialized by $U(t)$.
Let $C^k(\mathbb{R},H^s(\hat{\Sigma}_t, S\hat{M}|_{\hat{\Sigma}_t}))$ be the space of $C^k$-sections of this bundle. 
This space carries a natural topology induced by the semi-norms
\begin{equation}
\psi \mapsto \max_{t\in I}  \Big(\norm{(U\psi)(t)}_{H^s} + \ldots + \norm{\big(\tfrac{\dt}{\dt{t}}\big)^k(U\psi)(t)}_{H^s}\Big)
\label{eq:seminorms}
\end{equation}
where $I\subset\mathbb{R}$ runs through all compact intervals.
Here we wrote $(U\psi)(t)=U(t)\psi(t)$.

\begin{bem}
	Since $\nu$ is tangential to $\partial \hat{M}$, restriction to the boundary gives us isomorphisms
	\[
	U(t)|_{\partial\hat{\Sigma}_t}\colon H^s(\partial\hat{\Sigma}_t,S\hat{M}|_{\partial\hat{\Sigma}_t}))\to H^s(\partial\hat{\Sigma}_0,S\hat{M}|_{\partial\hat{\Sigma}_0}) .
	\]
\end{bem}
The following lemma, which is based on computations in \cite{vdD} and proven, for example, in Lemma~2.1 in \cite{SM} or in Section~4.1.2 in \cite{DGM}, reduces $\hat{D}$ to Hamiltonian form.
\begin{lem} \label{lem:ident}
	The operator $\hat{D}$ satisfies
	\[
	\hat{D}=-\hat{\gamma}(\hat{\nu}) U(t)^{-1}(\partial_t+i\tilde{D}_t)U(t),
	\]
	where $\tilde{D}_t:=U(t)\hat{D}_t U(t)^{-1}$.
	\hfill\qed
\end{lem}
The operator $\tilde{{D}}_t$ has the principal symbol
\begin{align}\label{identOpPrin}
	\sigma_{\tilde{{D}}_t}(\zeta)=\tau_t\circ \sigma_{\hat{{D}_t}}(\tau^{-1}_t\zeta)\circ \tau^{-1}_t
\end{align}
which is skew symmetric and satisfies the Clifford relations on $S\hat{M}|_{\hat{\Sigma}}$.
Since $U$ is an isometry on the $L^2$-spaces, $\tilde{{D}}_t$ is again formally selfadjoint.
Hence, $\tilde{D}_t$ is again a formally selfadjoint Dirac-type operator.
Moreover, the coefficients of $\tilde{D}_t$ depend smoothly on $t$.

\subsection{Standard setup} \label{Standard}
Unless stated otherwise, we will, from now on, assume to be in the following setting:
\begin{enumerate}[label=(S\arabic*), labelwidth=!, labelindent=2pt, leftmargin=21pt]
	\item\label{ssfirst}
	$(M,g)$ is a globally hyperbolic spatially compact spin manifold with timelike boundary $\partial M$.
	The inward pointing unit normal field to $\partial M$ is denoted by $\eta$.
	\item 
	$T\colon M\to \R$ is a Cauchy temporal function whose gradient is tangential to $\partial M$.
	The induced lapse function, Cauchy hypersurfaces and pastdirected unit normal field are denoted by $N$, $\Sigma_t$ and $\nu$, respectively.
	\item 
	$D\colon C^\infty(M,SM)\to C^\infty(M,SM)$ is the Lorentzian spin Dirac operator.
	\item
	$D_t\colon C^\infty(\Sigma_t,SM|_{\Sigma_t}) \to C^\infty(\Sigma_t,SM|_{\Sigma_t})$ is the induced Dirac operator on $\Sigma_t$ given by \eqref{eq:Dt}.
	\item
	$\tilde{D}_t\colon C^\infty(\Sigma_0,SM|_{\Sigma_0}) \to C^\infty(\Sigma_0,SM|_{\Sigma_0})$ is the induced Dirac-type operator described in Lemma~\ref{lem:ident}.
	\item
	$A_t\colon C^\infty(\partial\Sigma_t,SM|_{\partial\Sigma_t}) \to C^\infty(\partial\Sigma_t,SM|_{\partial\Sigma_t})$ is the boundary operator for $D_t$ given in \eqref{eq:closedDirac}.
	\item\label{sslast}
	$\tilde{A}_t\colon C^\infty(\partial\Sigma_0,SM|_{\partial\Sigma_0}) \to C^\infty(\partial\Sigma_0,SM|_{\partial\Sigma_0})$ is a boundary operator for $\tilde{D}_t$.
	It is assumed to be formally selfadjoint, to anticommute with $\sigma_{\tilde{{D}}_t}(\hat\eta(0)^\flat)$ and to have coefficients depending smoothly on $t$.
\end{enumerate}
Using \eqref{conformalOpPrin} and \eqref{identOpPrin}, we compute the principal symbol of $\tilde{A}_t$:
\begin{align*}
	\sigma_ {\tilde{A}_t}(\zeta)
	&=
	\sigma_{\tilde{D}_t}(\hat{\eta}(0)^\flat)^{-1}\sigma_{\tilde{D}_t}(\zeta) \\
	&=
	(\tau_t \sigma_{\hat{D}_t}(\tau_t^{-1}\hat{\eta}(0)^\flat)\tau_t^{-1})^{-1}(\tau_t \sigma_{\hat{D}_t}(\tau_t^{-1}\zeta)\tau_t^{-1}) \\
	&=
	\tau_t \sigma_{D_t}(\tau_t^{-1}\hat{\eta}(0)^\flat)^{-1}\sigma_{D_t}(\tau_t^{-1}\zeta)\tau_t^{-1} .
\end{align*}
\begin{bem}\label{boprincipalsymbol}
If we assume additionally that $N|_{\partial M}$ only depends on $t$ and that $\eta$ is parallel along the integral curves of $\partial_t$, then
$$
\tau_t^{-1}\hat{\eta}(0)^\flat =N(0)^{-1} \tau_t^{-1}\eta(0)^\flat =N(0)^{-1} \eta(t)^\flat
$$
and hence
	\[
	\sigma_{\tilde{A}_t}(\zeta)
	= N(0)
	\tau_t\sigma_{D_t}(\eta(t)^\flat)^{-1}\sigma_{D_t}(\tau_t^{-1}\zeta)\tau_t^{-1}
	= N(0)
	\tau_t\sigma_{A_t}(\tau_t^{-1}\zeta)\tau_t^{-1}.
	\]
	Thus in this case we can make the choice $\tilde{A}_t=N(0)U(t)A_tU(t)^{-1}$.
\end{bem}

\subsection{Admissible boundary conditions}\label{BoundaryConditions}
We now apply the results of Section~\ref{ConFunCal} to the family of elliptic formally selfadjoint Dirac-type operators $\tilde{D}_t$.

\begin{defi}\label{admissibleBound}
	Suppose we are in the Standard Setup~\ref{ssfirst}--\ref{sslast}.
	We call a family $ B=\{B_t\}_{t\in\mathbb{R}}$ of linear subspaces $B_t\subseteq H^{\frac 12}(\Sigma_t,SM|_{\Sigma_t})$ an \emph{admissible boundary condition} for $D$ (w.r.t.\ the temporal function $T$) if
	\begin{enumerate}[label=(\arabic*), labelwidth=!, labelindent=2pt, leftmargin=21pt]
		\item \label{admissibleBound1}
		$N^{\nicefrac 12}(t,\cdot)B_t = \{N^{\nicefrac 12}(t,\cdot)u; u\in B_t\}\subset H^{\frac 12}(\Sigma_t,SM|_{\Sigma_t})$ is a selfadjoint boundary condition for $D_t$ for every $t\in\mathbb{R}$;
		\item \label{admissibleBound2}
		$\tilde{B}_t:=U(t)N^{\frac{ n}{2}}(t,\cdot)B_t\subseteq H^{\frac 12}(\Sigma_0,SM|_{\Sigma_0})$ is a selfadjoint $\infty$-regular boundary condition for $\tilde{D}_t$ for every $t\in\mathbb{R}$;
		\item \label{admissibleBound3}
		$t\mapsto\tilde{D}_{t,\tilde{B}_t}J_{\tilde{B}_t}^{(\varepsilon)}$ is strongly continuous with respect to $\norm{\cdot}_{H^k}$ for every $k\in \mathbb{N}$.
	\end{enumerate}
\end{defi}
Note that the definition depends crucially on the choice of temporal function $T$.

\medskip
Let us discuss a simple example for illustrating Definition \ref{admissibleBound} and see that indeed such boundary conditions exist:
\begin{bsp} \label{Lor:trans}
	Let us discuss the Lorentzian equivalent of the transmission conditions we considered in Example~\ref{bsp:trans}.
	We start with a global hyperbolic spin manifold $(M,g)$ with closed Cauchy hypersurfaces.
	Let $\mathcal{M}\subseteq M$ be a timelike hypersurface with trivial normal bundle.
	Cut $M$ along $\mathcal{M}$ to obtain a globally hyperbolic spacetime $M'$ with timelike boundary $\partial M'=\mathcal{M}_1\sqcup \mathcal{M}_2$.
	We can then choose a temporal function $T\colon M' \to\mathbb{R}$ with gradient tangential to $\mathcal{M}$.
	This implies that $\partial\Sigma'_t=(\mathcal{M}_1)_t\sqcup (\mathcal{M}_2)_t$ with $(\mathcal{M}_1)_t$ and $(\mathcal{M}_2)_t$ being canonically diffeomorphic to $\mathcal{M}_t$ .
	This allows us to consider a family $B=\{B_t\}_{t\in \mathbb{R}}$ with
	\begin{align*}
	B_t:=\big\{(\psi,\psi)\in H^{\frac 12}((\mathcal{M}_1)_t, SM'|_{(\mathcal{M}_1)_t})\oplus H^{\frac 12}((\mathcal{M}_2)_t,& SM'|_{(\mathcal{M}_2)_t});\\
    & \psi\in H^{\frac 12}(\mathcal{M}_t,SM|_{\mathcal{M}_t})\big\}
	\end{align*}
	being a transmission condition as defined in Example~\ref{bsp:trans}.
	Hence, $B_t$ is a selfadjoint and $\infty$-regular boundary condition for $D_t$.
	Furthermore, for any $k\in\mathbb{R}$,
	\begin{align}\label{equiv:trans}
	N^{\nicefrac k2}(t)B_t=\{N^{\nicefrac k2}(t,\cdot)u; u\in B_t\}=B_t
	\end{align}
	since $N^{\nicefrac{k}{2}}(t,\cdot)H^\frac 12=H^\frac 12$. 
	Hence, $B_t=N^{\nicefrac 12}(t,\cdot)B_t$ is a selfadjoint boundary condition for $D_t$ as well, which implies Assumption~\ref{admissibleBound1} in Definition~\ref{admissibleBound}.
	Using again \eqref{equiv:trans}, we can rewrite $\tilde{B}_t$ by
	\begin{align*}
	\tilde{B}_t&=U(t)N^{\nicefrac n 2}(t,\cdot)B_t  \\
	&=U(t)B_t \\
	&=\big\{(U(t)\psi,U(t)\psi)\in H^{\frac 12}((\mathcal{M}_1)_0, SM'|_{(\mathcal{M}_1)_0})\oplus H^{\frac 12}((\mathcal{M}_2)_0, SM'|_{(\mathcal{M}_2)_0});\\
    & \quad\quad \psi\in H^{\frac 12}(\mathcal{M}_t,SM|_{\mathcal{M}_t})\big\} \\
	&=B_0
	\end{align*}
	since  $U(t):H^{\frac12}(\Sigma'_t)\to H^{\frac 12}(\Sigma'_0)$ is an isomorphism.
	Hence,  $\tilde{B}_t=B_0$ is a selfadjoint $\infty$-regular boundary condition for $\tilde{D}_t$, which implies Assumption~\ref{admissibleBound2} in Defintion~\ref{admissibleBound}.
	Assumption~\ref{admissibleBound3} in Definition~\ref{admissibleBound} follows from the fact that $D_t$ has smooth coefficients and $\tilde{B}_t=B_0$ is independent of $t$.

\end{bsp}

Given an admissible boundary condition $B$, we have the corresponding \emph{Lorentzian boundary condition}
\[
C^\infty(\partial M,B):=\{\psi\in C^\infty(\partial M, SM|_{\partial M}); \psi|_{\partial\Sigma_t}\in B_t ~\text{for every }t\in\mathbb{R}\}.
\]

\begin{lem} \label{selfadjoint}
	Assume the Standard Setup~\ref{ssfirst}--\ref{sslast} and let $B$ be an admissible boundary condition.
	Then
	\[
	\int_{M} \spann{D\psi,\phi}_{SM}+\spann{\psi,D\phi}_{SM}\dt{\mu}_M=0
	\]
	for all $\psi,\phi\in C^\infty_\c(M,SM)$ with $\psi|_{\partial M},\phi|_{\partial M}\in C^\infty(\partial M,B)$.
\end{lem}
\begin{proof}
	Using Fubini's Theorem, we can compute the boundary term in \eqref{Green} as follows:
	\begin{align*}
		\int_{\partial M}\spann{\gamma(\eta)\psi,\phi}_{SM}\dt{\mu}_{\partial M}
		&=
		\int_\R \int_{\partial\Sigma_t} \spann{\gamma(\eta_t)\psi,\phi}_{SM}N(t)\dt{\mu}_{\partial\Sigma_t}\dt{t} \\
		&=
		-i\int_{\mathbb{R}}\underbrace{\int_{\partial\Sigma_t}\spann{\sigma_{D_t}(\eta_t^\flat)N^{\nicefrac 1 2}(t)\psi,N^{\nicefrac 12}(t)\phi}_0\dt{\mu}_{\partial\Sigma_t}}_{=:(\ast_t)}\dt{t}  \\
		&=0,
	\end{align*}
where we used in the last line that $\eta_t^\flat$ is the conormal to $\partial\Sigma_t$ and thus $(\ast_t)$
 is taking the role of the RHS of \eqref{RiemGreen} with $\ngm=\eta_t^\flat$.
Since $N^{\nicefrac 12}\psi|_{\partial\Sigma_t},N^{\nicefrac 12}\phi|_{\partial\Sigma_t}\in N^{\nicefrac 12}(t,\cdot)B_t$, expression $(\ast_t)$  vanishes for all $t$ by Condition~\ref{admissibleBound1} in Definition~\ref{admissibleBound}.
\end{proof}

We will often use the norm equivalence of $\norm{\cdot}_{\tilde{B}_t,k}$ and $\norm{\cdot}_{H^k}$. 
We define for $k\in\mathbb{N}$ and $s>\frac12$
\[
C^k(\mathbb{R},H^s_{\tilde{B}_\bullet}(\hat{\Sigma},S\hat{M}|_{\hat{\Sigma}}))
:=
\{\psi\in C^k(\mathbb{R},H^s(\hat{\Sigma},S\hat{M}|_{\hat{\Sigma}}));~ \psi(t)\in H^s_{\tilde{B}_t}(\hat{\Sigma},S\hat{M}|_{\hat{\Sigma}})\}
\]
equipped with the subspace topology.
Thus this topology is given by the same seminorms as in \eqref{eq:seminorms}.
Furthermore,  $C^k(I,H^s_{\tilde{B}_\bullet}(\hat{\Sigma},S\hat{M}|_{\hat{\Sigma}}))$ is a Banach space for any compact interval $I$.
In this case, we will denote the norms by
\[
\norm{\psi}_{k,I,H^s}
=
\max_{t\in I}  \Big(\norm{(U\psi)(t)}_{H^s} + \ldots + \norm{\big(\tfrac{\dt}{\dt{t}}\big)^k(U\psi)(t)}_{H^s}\Big) .
\]
The condition $s>\frac12$ is needed for $\psi|_{\partial M}$ to make sense.

\section{Initial Boundary Value Problems}
\label{sec:IBVP}

We now discuss the Cauchy problem under boundary conditions as described in Definition~\ref{admissibleBound}. 
Throughout this section, we will assume the Standard Setup~\ref{ssfirst}--\ref{sslast}.

\medskip
Let  $B=\{B_t\}_{t\in\mathbb{R}}$ be an admissible boundary condition as described in Definition~\ref{admissibleBound}. 
Consider the following Cauchy problem:
\begin{align}\label{CP:1}
	\begin{cases}
		D\psi=f\in C^\infty_{\cc}(M,SM), &\\
		\psi|_{\Sigma_0}=\psi_0\in C^\infty_{\cc}(\Sigma_0,SM|_{\Sigma_0}), & \\
		\psi|_{\partial M}\in C^\infty(\partial M,B). & 
	\end{cases}
\end{align}
Here $f$ and $\psi_0$ are given and $\psi$ is searched for.
The main result of this section is the well-posedness of this Cauchy problem:

\begin{satz}\label{main}
Assume the Standard Setup~\ref{ssfirst}--\ref{sslast} and let $B$ be an admissible boundary condition.

Then, given $f\in C^\infty_{\cc}(M,SM)$ and $\psi_0\in C^\infty_{\cc}(\Sigma_0,SM|_{\Sigma_0})$, there exists a unique smooth solution $\psi\in C^\infty(M,SM)$ to the Cauchy problem~\eqref{CP:1}.
This solution depends continuously on the Cauchy data $(f,\psi_0)$.
\end{satz}

Uniqueness will follow from an $L^2$-estimate which we derive in Subsection~\ref{sec:uni}.
Existence will be proved by regularizing the problem and passing to a limit using the Arzela-Ascoli Theorem in Subsection~\ref{sec:exist}.
In Subsection~\ref{sec:cont} we will show continuous dependence on the Cauchy data.

Before we start, we apply the reduction we did in Section~\ref{CauchyId}.
Then $\psi\in C^\infty(M,SM)$ is a solution to the Cauchy problem~\eqref{CP:1} if and only if $\tilde{\psi}=UN^{\frac n 2}\psi$ is a solution to  the following Cauchy problem
\begin{align}\label{CP:2}
	\begin{cases}
		\tilde{D}\tilde\psi=\tilde{f}\in C^\infty_\c(\mathbb{R},C^\infty_{\cc}(\hat{\Sigma}_0,S\hat{M}|_{\hat{\Sigma}_0})), & \\
		\tilde{\psi}(0)=N^{\frac n 2}(0) \psi_0\in C^\infty_{\cc}(\hat{\Sigma_0},S\hat{M}|_{\hat{\Sigma}_0}),& \\
		\tilde{\psi}|_{\partial \hat{\Sigma}_0}\in C^\infty(\mathbb{R},\tilde{B}),
	\end{cases}
\end{align}
where $\tilde{f}=UN^{\frac{n+2}{2}}f$ and 
\[
C^\infty(\mathbb{R},\tilde{B}):=\{\psi\in C^\infty(\mathbb{R},C^\infty(\partial\hat{\Sigma}_0,S\hat{M}|_{\partial\hat{\Sigma}_0}));~ \psi(t)\in\tilde{B}_t~\forall \,t\in\mathbb{R}\}.
\]
Furthermore, note that $\tilde{\psi}$ is a solution to the Cauchy problem~\eqref{CP:2} if and only if it solves
\begin{align}\label{CP:3}
	\begin{cases}
		-\hat{\gamma}(\hat{\nu})\tilde{D}=(\partial_t+i\tilde{D}_t)\tilde{\psi}=-\hat{\gamma}(\hat{\nu}) \tilde{f}=:\check{f}, & \\
		\tilde{\psi}(0)=N^{\frac n 2}(0)\psi_0, & \\
		\tilde{\psi}|_{\partial \hat{ \Sigma}_0}\in C^\infty(\mathbb{R},\tilde{B}).&
	\end{cases}
\end{align}

\subsection{Energy estimate and uniqueness} \label{sec:uni}

Uniqueness will follow from an $L^2$-energy estimate. 
For $t_0,t_1\in\mathbb{R}$ we will write $M_{[t_0,t_1]}:=T^{-1}([t_0,t_1])$ and $(\partial M)_{[t_0,t_1]}:=M_{[t_0,t_1]}\cap \partial M$. 
\begin{prop}\label{EE} 
Assume the Standard Setup~\ref{ssfirst}--\ref{sslast} and let $B$ be an admissible boundary condition. 
Then for all $t_0,t_1\in\mathbb{R}$ with $t_1>t_0$, there exists a constant $C=C[t_0,t_1]$ such that
	\begin{align*}
		\int_{\Sigma_{t_1}}\abs{\psi}^2_0\dt{\mu}_{\Sigma_{t_1}}\leq e^{C(t_1-t_0)}\left[C\int_{t_0}^{t_1}\int_{\Sigma_{s}}\abs{D\psi}^2_0\dt{\mu}_{\Sigma_s}\dt{s}+\int_{\Sigma_{t_0}}\abs{\psi}^2_0 \dt{\mu}_{\Sigma_{t_0}}\right]
	\end{align*}
	for all $\psi\in C^\infty(M,SM)$ satifying $\psi|_{\partial M}\in C^\infty(\partial M,B)$.
\end{prop}
\begin{proof}
	The divergence theorem applied to the manifold $M_{[t_0,t_1]}$ with piecewise smooth boundary 
	\[
	\partial(M_{[t_0,t_1]})=(\partial M)_{[t_0,t_1]}\cup \Sigma_{t_1}\cup \Sigma_{t_0}
	\]
	yields
	\begin{align*}
		&\int_{M_{[t_0,t_1]}} \spann{D\psi,\psi}_{SM}+\spann{\psi,D\psi}_{SM}\dt{\mu}_M \\
		&=
		-\int_{(\partial M)_{[t_0,t_1]}}\spann{\gamma(\eta)\psi,\psi}_{SM}\dt{\mu}_{\partial M}+\int_{\Sigma_{t_1}}\abs{\psi}^2_0\dt{\mu}_{\Sigma_{t_1}}-\int_{\Sigma_{t_0}}\abs{\psi}^2_0\dt{\mu}_{\Sigma_{t_0}} .
	\end{align*}
	The left hand side can be estimated as
	\[
	\int_{M_{[t_0,t_1]}}\spann{D\psi,\psi}_{SM}+\spann{\psi,D\psi}_{SM}\dt{\mu}_M \leq C \int_{t_0}^{t_1}\int_{\Sigma_t}\abs{D\psi}^2_0+\abs{\psi}^2_0 \dt{\mu}_{\Sigma_t}\dt{t},
	\]
	where we used Cauchy-Schwarz, Fubini, boundedness of $N$ on the compact set $M_{[t_0,t_1]}$, and boundedness of $\gamma(\nu)$.
	The first term of the right hand side vanished by Lemma~\ref{selfadjoint}.
	We get in total
	\[
	\int_{\Sigma_{t_1}}\abs{\psi}_0^2\dt{\mu}_{\Sigma_{t_1}}\leq \int_{\Sigma_{t_0}}\abs{\psi}_0^2\dt{\mu}_{\Sigma_{t_0}}+C\int_{t_0}^{t_1}\int_{\Sigma_t}\abs{D\psi}^2_0+\abs{\psi}^2_0 \dt{\mu}_{\Sigma_t}\dt{t}.
	\]
	Gronwall's Lemma now gives us the desired estimate.
\end{proof}

This directly implies the following uniqueness statement in Theorem~\ref{main}:
\begin{koro}\label{uniqueness} 
Assume the Standard Setup~\ref{ssfirst}--\ref{sslast} and let $B$ be an admissible boundary condition. 
A section $\psi\in C^\infty(M,SM)$ with $\psi|_{\partial M}\in C^\infty(\partial M,B)$ is uniquely determined by $D\psi$ and by $\psi|_{\Sigma_0}$.
\end{koro}
\begin{proof}
	Let $t_1\in\mathbb{R}$ be arbitrary and put $t_0=0$. 
	If $\psi|_{\Sigma_0}=0$ and $D\psi=0$, then Proposition~\ref{EE} implies that $\psi|_{\Sigma_{t_1}}=0$. 
	Since $t_1$ is arbitrarily, $\psi\equiv 0$ follows.
\end{proof}

\subsection{Existence of smooth solutions} \label{sec:exist}
Next we show existence of smooth solutions to the Cauchy problem~\eqref{CP:1}.
We will regularize the problem and then apply an Arzela-Ascoli argument to the resulting solutions.

	To show existence, we will show existence of smooth solutions to the Cauchy problem~\eqref{CP:3}. 
	Put
	\[
	C^\infty(\hat{M},\tilde{B}):=\{\psi\in C^\infty(\mathbb{R},S\hat{M}|_{\hat{\Sigma}_0}); \psi(t)|_{\partial\hat{\Sigma}_0}\in \tilde{B}_t\}.
	\]
	 In the following, we will only work on $S\hat{M}|_{\hat{\Sigma}_0}$, so let us shorten the notation by setting $H^k:=H^k(\hat{\Sigma}_0,S\hat{M}|_{\hat{\Sigma}_0})$ and $H^k_{\tilde{B}_t}:=H^k_{\tilde{B}_t}(\hat{\Sigma}_0,S\hat{M}|_{\hat{\Sigma}_0})$.

\smallskip
\emph{(1.) Regularized problem} 
		
		\smallskip
		For any $\varepsilon>0$ and $k\in \mathbb{N}$, the regularized problem
		\begin{align}\label{regPro}
			(\partial_t+i\tilde{D}_{t,\tilde{B}_t}J_{\tilde{B}_t}^{(\varepsilon)})\tilde{\psi}^{(\varepsilon)}(t)=\check{f}(t), 
		\end{align}
		with initial condition $\tilde{\psi}^{(\varepsilon)}(0)=N^{\frac n 2}(0)\psi_0$ has a unique solution $\tilde{\psi}^{(\varepsilon),k}\in C^1(\mathbb{R},H^k_{\tilde{B}_\bullet})$.  
		This follows from Theorem~X.69 in \cite{ReedSimon2} and the fact that $\{\tilde{D}_{t,\tilde{B}_t}J_{\tilde{B}_t}^{(\varepsilon)}\}_{t\in\mathbb{R}}$ is a strongly continuous family of bounded operators from $H^k_{\tilde{B}_t}$ to $H^k_{\tilde{B}_t}$.
		Since $H^k_{\tilde{B}_t}\subseteq H^l_{\tilde{B}_t}$ for $l<k$, uniqueness of the solutions implies that $\tilde{\psi}^{(\varepsilon)}:= \tilde{\psi}^{(\varepsilon),k}$ is independent of $k$. 
		In particular, $\tilde{\psi}^{(\varepsilon)}$ is smooth in spatial directions.
		
		\smallskip
\emph{(2.) Requirements of the Arzela-Ascoli theorem}

		\smallskip
		We will derive bounds on the growth of $\tilde{\psi}^{(\varepsilon)}\in C^1(\mathbb{R},H^k_{\tilde{B}_\bullet})$ in time.
		The important fact will be that the constants $c_j$ ocurring in the estimates below do not depend on $u,u_0,f$ and $\varepsilon$. 
		They do depend on time $t$, but continuously, hence they are bounded on compact time intervals. 
		We compute the following $t$-derivative for any $l\in\mathbb{N}_0$:
		 
		 \begin{align*}
		 \frac{\dt{}}{\dt{t}}\norm{\tilde{D}^l_{t,\tilde{B}_t}\tilde{\psi}^{(\varepsilon)}(t)}_{L^2}^2
		 &=
		 2\Rea\spann{\partial_t\tilde{D}^l_{t,\tilde{B}_t}\tilde{\psi}^{(\varepsilon)}(t),\tilde{D}^l_{t,\tilde{B}_t}\tilde{\psi}^{(\varepsilon)}(t)}_{L^2} \\
		 &=
		 2\Rea \left(\spann{[\partial_t,\tilde{D}^l_{t,\tilde{B}_t}]\tilde{\psi}^{(\varepsilon)}(t),\tilde{D}^l_{t,\tilde{B}_t}\tilde{\psi}^{(\varepsilon)}(t)}_{L^2} \right. \\
		 &\quad
		 \left.+\spann{\tilde{D}^l_{t,\tilde{B}_t}\partial_t\tilde{\psi}^{(\varepsilon)}(t),\tilde{D}^l_{t,\tilde{B}_t}\tilde{\psi}^{(\varepsilon)}(t)}_{L^2}  \right) .
		 \end{align*}
		 A priori, the commutator $[\partial_t,\tilde{D}^l_{t,\tilde{B}_t}]$ is a differential operator of order $l+1$, but its principal symbol vanishes,
		 \[
		 \sigma_ {[\partial_t,\tilde{D}^l_{t,\tilde{B}_t}]}(\zeta)=[\sigma_{\partial_t}(\zeta),\sigma_{\tilde{D}_t^k}(\zeta)]=0,
		 \]
		because the principal symbol $\sigma_{\partial_t}(\zeta)=\zeta(\partial_t)$ is scalar.
		Hence, $[\partial_t,\tilde{D}^l_{t,\tilde{B}_t}]$ is of order~$l$ and we can bound the $L^2$-norm of the commutator by $\norm{\tilde{\psi}^{(\varepsilon)}(t)}_{\tilde{B}_t,l}$.
		Since $\psi^{(\varepsilon)}$ solves \eqref{regPro}, we find
		\begin{align*}
		\frac{\dt{}}{\dt{t}}\norm{\tilde{D}^l_{t,\tilde{B}_t}\tilde{\psi}^{(\varepsilon)}(t)}_{L^2}^2
		&\leq 
		c_1\norm{\psi^{(\varepsilon)}(t)}^2_{\tilde{B}_t,l}+2\Rea\spann{\tilde{D}^l_{t,\tilde{B}_t}\partial_t\tilde{\psi}^{(\varepsilon)}(t),\tilde{D}^l_{t,\tilde{B}_t}\tilde{\psi}^{(\varepsilon)}(t)}_{L^2} \\
		&\leq
		c_1 \norm{\tilde{\psi}^{(\varepsilon)}(t)}^2_{\tilde{B}_t,l} + 2\Rea\spann{\tilde{D}^l_{t,\tilde{B}_t}\check{f}(t),\tilde{D}^l_{t,\tilde{B}_t}\tilde{\psi}^{(\varepsilon)}(t)}_{L^2} \\
		&\quad
		+2\operatorname{Im}\spann{\tilde{D}^{l+1}_{t,\tilde{B}_t}J_{\tilde{B}_t}^{(\varepsilon)}\tilde{\psi}^{(\varepsilon)}(t),\tilde{D}^l_{t,\tilde{B}_t}\psi^{(\varepsilon)}(t)}_{L^2} \\
		&\leq
		c_1 \norm{\tilde{\psi}^{(\varepsilon)}(t)}^2_{\tilde{B}_t,l} +  \norm{\tilde{D}^l_{t,\tilde{B}_t}\check{f}(t)}^2_{L^2}+\norm{\tilde{D}^l_{t,\tilde{B}_t}\tilde{\psi}^{(\varepsilon)}(t)}^2_{L^2}\\
		&\quad
		+2\operatorname{Im}\spann{\tilde{D}^{l+1}_{t,\tilde{B}_t}J_{\tilde{B}_t}^{(\varepsilon)}\tilde{\psi}^{(\varepsilon)}(t),\tilde{D}^l_{t,\tilde{B}_t}\psi^{(\varepsilon)}(t)}_{L^2}.
		\end{align*}
		The last term vanishes because $J_{\tilde{B}_t}^{(\nicefrac{\varepsilon}{2})}$ and $\tilde{D}_{t,\tilde{B}_t}$ commute and are selfadjoint so that
		\begin{align*}
			\operatorname{Im}\spann{\tilde{D}^{l+1}_{t,\tilde{B}_t}J_{\tilde{B}_t}^{(\varepsilon)}\tilde{\psi}^{(\varepsilon)}(t),\tilde{D}^l_{t,\tilde{B}_t}\psi^{(\varepsilon)}(t)}_{L^2} 
			=
			\operatorname{Im}\spann{\tilde{D}^{l+1}_{t,\tilde{B}_t}J_{\tilde{B}_t}^{(\nicefrac{\varepsilon}{2})}\tilde{\psi}^{(\varepsilon)}(t),\tilde{D}^{l}_{t,\tilde{B}_t}J_{\tilde{B}_t}^{(\nicefrac{\varepsilon}{2})}\tilde{\psi}^{(\varepsilon)}(t)}_{L^2} 
			=0.
		\end{align*}
		Hence
		\[
		\frac{\dt{}}{\dt{t}}\norm{\tilde{D}^l_{t,\tilde{B}_t}\tilde{\psi}^{(\varepsilon)}(t)}_{L^2}^2
		\leq
		\norm{\tilde{D}^l_{t,\tilde{B}_t}\check{f}(t)}^2_{L^2}+c_1\norm{\tilde{\psi}^{(\varepsilon)}(t)}^2_{\tilde{B}_t,l}+\norm{\tilde{D}^l_{t,\tilde{B}_t}\tilde{\psi}^{(\varepsilon)}(t)}^2_{L^2}.
		\]
		Adding this inequality for $l=0$ and $l=k$, we find
		\[
		\frac{\dt{}}{\dt{t}}	\norm{\tilde{\psi}^{(\varepsilon)}(t)}^2_{\tilde{B}_t,k}\leq  \norm{\check{f}(t)}^2_{\tilde{B}_t,k}+c_2\norm{\tilde{\psi}^{(\varepsilon)}(t)}^2_{\tilde{B}_t,l}.
		\]
		Gronwall's Lemma yields for $t>0$
		\begin{align*}
		\norm{\tilde{\psi}^{(\varepsilon)}(t)}^2_{\tilde{B}_t,k}
		&\leq
		\norm{\psi^{(\varepsilon)}(0)}^2_{\tilde{B}_t,k}\exp\left(\int_0^t c_2(s)\dt{s}\right)+\int_0^t\norm{\check{f}(s)}^2_{\tilde{B}_t,k}\exp\left(\int_s^t c_2(\sigma)\dt{\sigma}\right)\dt{s} \\
		&\leq
		c_3\left[\norm{\psi^{(\varepsilon)}(0)}^2_{\tilde{B}_t,k}+\int_0^t\norm{\check{f}(s)}^2_{\tilde{B}_t,k}\dt{s}\right] \\
		&=
		c_3\left[\norm{N^{\nicefrac{n}{2}}(0)\psi_0}^2_{\tilde{B}_t,k}+\int_0^t\norm{\check{f}(s)}^2_{\tilde{B}_t,k}\dt{s}\right]
		\end{align*}
		and similarly for $t<0$.
		Thus, for fixed $t\in \mathbb{R}$, the set $\{\tilde{\psi}^{(\varepsilon)}(t);\varepsilon>0\}$ is bounded in $H^{k}_{\tilde{B}_t}$ for all $k$.
		By the Rellich-Kondrachov theorem, $\{\psi^{(\varepsilon)}(t);~\varepsilon>0\}$ is relatively compact in $H^k$ for all $k$ and $t\in \mathbb{R}$.
		
		Taking into account that $J_{\tilde{B}_t}^{(\varepsilon)}$ is a contraction on $H^{k+1}_{\tilde{B}_t}$ we get
		\begin{align*}
			\norm{\partial_t \tilde{\psi}^{(\varepsilon)}(t)}_{\tilde{B}_t,k}
			&=
			\norm{\check{f}(t)-i\tilde{D}_{t,\tilde{B}_t}J_{\tilde{B}_t}^{(\varepsilon)}\tilde{\psi}^{(\varepsilon)}(t)}_{\tilde{B}_t,k} \\
			&\leq
			\norm{\check{f}(t)}_{\tilde{B}_t,k}+c_4 \norm{J_{\tilde{B}_t}^{(\varepsilon)}\tilde{\psi}^{(\varepsilon)}(t)}_{\tilde{B}_t,k+1} \\
			&\leq
			\norm{\check{f}(t)}_{\tilde{B}_t,k}+c_4 \norm{\tilde{\psi}^{(\varepsilon)}(t)}_{\tilde{B}_t,k+1} \\
			&\leq
			c_5,
		\end{align*}
		where $c_5$ still does not depend on $\varepsilon$.
		Using that $\norm{\cdot}_{\tilde{B}_{t},k}\simeq \norm{\cdot}_k$ on $H^k_{\tilde{B}_t}$, we see that $t\mapsto \tilde{\psi}^{(\varepsilon)}(t)$ is equicontinuous. 

		\smallskip
	 \emph{(3.) Arzela-Ascoli theorem} 
	 
	 \smallskip
		For $k,l\in\mathbb{N}$ the Arzela-Ascoli theorem implies that $\{\tilde{\psi}^{(\varepsilon)};\varepsilon >0\}\subseteq C^0([-l,l],H^k)$ is relatively compact.
		Thus we obtain a sequence $\varepsilon_j\to0$ such that $\tilde{\psi}^{(\varepsilon_j)}\to\tilde{\psi}\in C^0([-l,l],H^k)$ as $j\to \infty$.
		By a diagonal subsequence argument we can assume that $\tilde{\psi}^{(\varepsilon_j)}\to \tilde{\psi}\in C^0([-l,l],H^k)$ for all $k,l\in\mathbb{N}$.
		Therefore the convergence $\tilde{\psi}^{(\varepsilon_j)}\to\psi$ is locally uniform in $C^0(\mathbb{R},H^k)$ for all $k$. 
		In particular, $\tilde{\psi}^{(\varepsilon_j)}(t)\to\tilde{\psi}(t)$ in $H^k$ for every $t\in\mathbb{R}$.
		Since $H^k_{\tilde{B}_t}\subseteq H^k$ closed, this implies $\tilde{\psi}(t)\in H^k_{\tilde{B}_t}$ for every $t$.
		Thus, $\tilde{\psi}$ satisfies the Lorentzian boundary condition.

		\smallskip
	\emph{(4.) Solution to the Cauchy problem~\eqref{CP:3}}
		
		\smallskip
		For $t=0$ we have $\tilde{\psi}^{(\varepsilon)}(0)=N^{\frac n 2}(0)\psi_0$ for every $\varepsilon$ and hence $\tilde{\psi}(0)=N^{\frac n 2}(0)\psi_0$.
		To show $\tilde{D}\tilde{\psi}=\check{f}$ is a bit more involved because $\tilde{D}$ contains a time derivative and so far we only have $C^0$-convergence in time as $\varepsilon_j\to 0$.
		In order to get rid of the time derivatives, we integrate the regularized problem~\eqref{regPro} and we obtain
		\begin{align}\label{main:eq2}
			\tilde{\psi}^{(\varepsilon_j)}(t)-N^{\frac n 2}(0)\psi_0=\int_0^t \left[-i\tilde{D}_{s,\tilde{B}_s}J_{\tilde{B}_s}^{(\varepsilon_j)} \tilde{\psi}^{(\varepsilon_j)}(s)+\check{f}(s)\right]\dt{s}.
		\end{align}
		Now we let $\varepsilon_j\to 0$.
		For the left hand side of~\eqref{main:eq2} we find $\tilde{\psi}^{(\varepsilon_j)}(t)-N^{\frac n 2}(0)\psi_0\to \tilde{\psi}(t)-N^{\frac n 2}(0)\psi_0$.
		For the right hand side of \eqref{main:eq2} we consider the first summand under the integral which is the one depending on $\varepsilon_j$.
		It converges pointwise
		\begin{align*}
			\tilde{D}_{s,\tilde{B}_s}J_{\tilde{B}_s}^{(\varepsilon_j)} \tilde{\psi}^{(\varepsilon_j)}(s)
			&=
			J_{\tilde{B}_s}^{(\varepsilon_j)} \tilde{D}_{s,\tilde{B}_s}\tilde{\psi}^{(\varepsilon_j)}(s) \\
			&=
			J_{\tilde{B}_s}^{(\varepsilon_j)} \tilde{D}_{s,\tilde{B}_s}\tilde{\psi}(s)
			+ J_{\tilde{B}_s}^{(\varepsilon_j)} \tilde{D}_{s,\tilde{B}_s}\big(\tilde{\psi}^{(\varepsilon_j)}(s) - \tilde{\psi}(s)\big) \\
			&\to
			\tilde{D}_{s,\tilde{B}_s}\tilde{\psi}(s) + 0.
		\end{align*}
		Here used that $J_{\tilde{B}_s}^{(\varepsilon_j)} \to  \id$ strongly, $J_{\tilde{B}_s}^{(\varepsilon_j)}$ is a contraction on $H^k_{\tilde{B}_s}$, and $\tilde{D}_{s,\tilde{B}_s}\big(\tilde{\psi}^{(\varepsilon_j)}(s) - \tilde{\psi}(s)\big)\to 0$ in $H^k_{\tilde{B}_s}$.
		Moreover, the integrand is bounded in $H^k$ because
		\begin{align*}
			\norm{\tilde{D}_{s,\tilde{B}_s}J_{\tilde{B}_s}^{(\varepsilon_j)}\tilde{\psi}^{(\varepsilon_j)}(s)}_k
			&=
			\norm{J_{\tilde{B}_s}^{(\varepsilon_j)} \tilde{D}_{s,\tilde{B}_s}\tilde{\psi}^{(\varepsilon_j)}(s)}_{k}
			\leq
			\norm{\tilde{D}_{s,\tilde{B}_s} \tilde{\psi}^{(\varepsilon_j)}(s)}_{k} .
		\end{align*}
		Dominanted convergence yields
		\[
		\tilde{\psi}(t)-N^{\frac n 2}(0)\psi_0
		=
		\int_0^t \left[-i\tilde{D}_{s,\tilde{B}_s} \tilde{\psi}(s)+\check{f}(s)\right]\dt{s}.
		\]
		Differentiating this equation shows that $\tilde{\psi}$ solves \eqref{CP:3}.

		\smallskip
	 \emph{(5.) Regularity} 
		
		\smallskip
		So far we know continuity of $\tilde{\psi}$ in time and smoothness in the spatial directions.
		From $\tilde{\psi}\in C^0(\mathbb{R},H^k)$ and by the equation in \eqref{CP:3} we find $\partial_t\tilde{\psi}\in C^0(\mathbb{R},H^{k-1})$ for all $k$.
		Hence $\tilde{\psi}\in C^1(\mathbb{R},H^{k})$ for all $k$.  
		For the second derivative, we differentiate the equation in \eqref{CP:3} in time and by comparing both sides of the equation, we find $\tilde{\psi}\in C^2(\mathbb{R},H^k)$ for all $k$.
		Iterating this argument, we see that $\tilde{\psi}\in C^\infty(\mathbb{R},C^\infty(\hat{\Sigma}_0,S\hat{M}|_{\hat{\Sigma}_0}))$.
		Hence we showed the existence of a smooth solution to the Cauchy problem~\eqref{CP:3} and equivalently to the Cauchy problem~\eqref{CP:1}.

\subsection{Well-posedness of the Cauchy problems}
\label{sec:cont}
In this subsection, we complete the proof of Theorem~\ref{main} by showing that the solution constructed in Subsection~\ref{sec:exist} depends continuously on the Cauchy data.
We write
\[
C^\infty(M,B)
=
\{\psi\in C^\infty(M,SM);~\psi|_{\partial M}\in C^\infty(\partial M,B)\}.
\]
Consider the map
\begin{align*}
P:=D\oplus \res_0\colon C^\infty(M,B)&\to C^\infty(M,SM)\oplus C^\infty(\Sigma,SM|_{\Sigma_0}), \\
\psi &\mapsto (D\psi,\psi|_{\Sigma_0}),
\end{align*}
which is clearly continuous and linear. 

Fix compact sets $A_1\subseteq M$ and $A_2\subseteq \Sigma_0$ such that $A_i\cap \partial M=\emptyset$.
Then define $C^\infty_{A_1}(M,SM)$ as the space of smooth sections with support contained in $A_1$ and similarly for $C^\infty_{A_2}(\Sigma_0,SM|_{\Sigma_0})$.
Then $C^\infty_{A_1}(M,SM)\oplus C^\infty_{A_2}(\Sigma_0, SM|_{\Sigma_0})$ is closed in $C^\infty_{\cc}(M,SM)\oplus C^\infty_{\cc}(\Sigma_0,SM|_{\Sigma_0})$.
Hence
\[
\mathscr{V}_{A_1,A_2}
:=
P^{-1}\left(C^\infty_{A_1}(M,SM)\oplus C^\infty_{A_2}(\Sigma, SM|_\Sigma)\right)
\]
is closed in $C^\infty(M,B)$. 
In particular, all these spaces are Fréchet spaces.
By existence and uniqueness of the solution, $P$ maps $\mathscr{V}_{A_1,A_2}$ bijectively onto $C^\infty_{A_1}(M,SM)\oplus C^\infty_{A_2}(\Sigma, SM|_\Sigma)$.
The open mapping theorem for Fréchet spaces (see Theorem~V.6 in \cite{ReedSimon1}) then gives us that
\[
(P|_{\mathscr{V}_{A_1,A_2}})^{-1}\colon C^\infty_{A_1}(M,SM)\oplus C^\infty_{A_2}(\Sigma, SM|_\Sigma)\to \mathscr{V}_{A_1,A_2} \subseteq C^\infty(M,B)
\]
is continuous as well.
Since $C^\infty_{\cc}(M,SM)$ and $C^\infty_{\cc}(\Sigma_0,SM|_{\Sigma_0})$ carry the direct limit topologies as we vary over $A_1$ and $A_2$, respectively, the solution map
\begin{align}
C^\infty_{\cc}(M,SM)\oplus C^\infty_{\cc}(\Sigma_0,SM|_{\Sigma_0})\to C^\infty(M,B),\quad
(f,\psi_0)\mapsto \psi,
\label{eq:solutionmap}
\end{align}
is also continuous.
This concludes the proof of Theorem~\ref{main}.

\begin{bem}\label{remarks} 
	\begin{enumerate}[label=(\alph*), wide, labelwidth=!, labelindent=0pt]
		\item In Cauchy problem~\eqref{CP:1}, the initial value $\psi_0$ was imposed on $\Sigma_0$.
		By shifting the $t$-parameter, one can equally well impose initial values on any time\-slice $\Sigma_{t_0}$.
		\item The solution map \eqref{eq:solutionmap} is not surjective.
		In fact, the solutions satisfy an additional finite propagation property as we will see in the next section.
		\item Theorem~\ref{main} can be extended to the case of noncompact Cauchy hypersurfaces, but with still compact boundary.
		This requires a covering argument as in \cite{Waves,DGM}, for example.
	\end{enumerate}
\end{bem}

\section{Finite propagation speed and Green operators} \label{sec:propagation}

We now study the support of solutions to the Cauchy problem~\eqref{CP:1} which, as usual for wave equations, essentially says that a wave propagates with the speed of light at most.
There is an interesting violation of this principle, however.
As soon as the wave hits the boundary somewhere, the \emph{whole} boundary radiates off instantenously (w.r.t.\ the given temporal function).
This is due to the global nature of the boundary conditions and violates the causal principle that no signal should propagate fast than with the speed of light.
We show by example that this effect really occurs.

Once we have proved the support properies, we will show existence of advanced and retarded Green's operators for the Dirac operator on a globally hyperbolic spacetime with timelike boundary.

\subsection{Finite propagation speed}
We start by analyzing the propagation of solutions to the Cauchy problem~\eqref{CP:1}.
\begin{prop}\label{prop:support}
	Assume the Standard Setup~\ref{ssfirst}--\ref{sslast} and let $B$ be an admissible boundary condition.
	Let $\psi\in C^\infty(M,SM)$ be the smooth solution to the Cauchy problem~\eqref{CP:1} for the Cauchy data $(f,\psi_0)$.
	Then the support of $\psi$ satisfies
	\begin{align}\label{support}
	J_{\pm}(\Sigma_0)\cap \operatorname{supp}(\psi)
	\subseteq
	J_{\pm}\big{(}\operatorname{supp}(f)\cap J_{\pm}(\Sigma_0)\cup \operatorname{supp}(\psi_0)\cup \partial\Sigma_{t_{\pm}}\big{)}
	\end{align}
	where 
	\begin{align*} 
		t_+&=\min T\big(J_+\big(\operatorname{supp}(f)\cap J_+(\Sigma_0)\cup \operatorname{supp}(\psi_0)\big)\cap\partial M\big), \\
		t_-&=\max  T\big(J_-\big(\operatorname{supp}(f)\cap J_-(\Sigma_0)\cup \operatorname{supp}(\psi_0)\big)\cap\partial M\big) .
	\end{align*}
\end{prop}
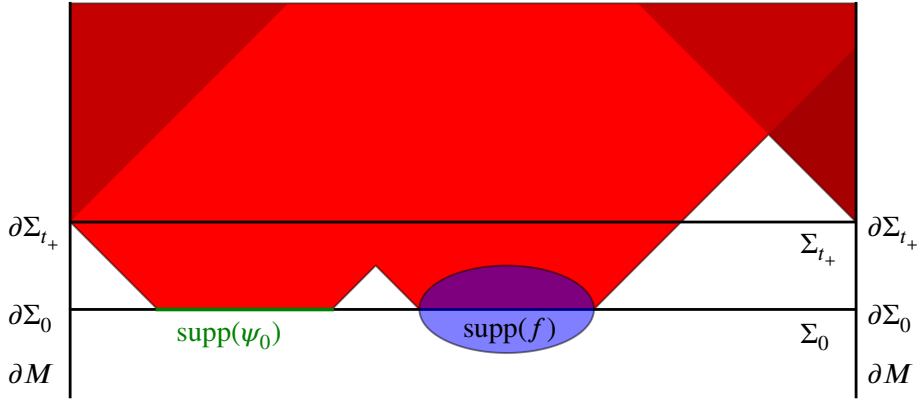
\begin{figure}[h]
	\begin{overpic}[scale=0.7]{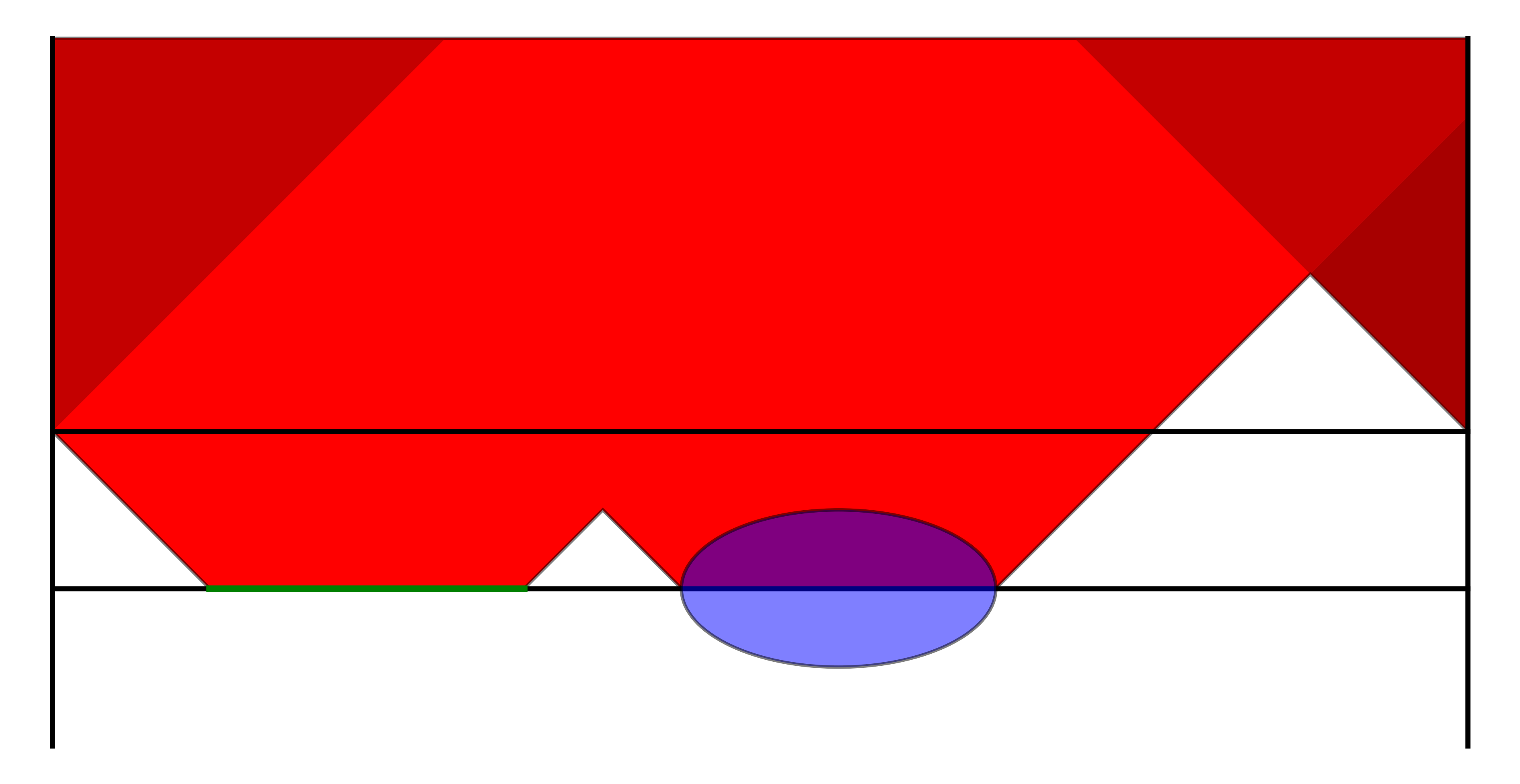}
		\put(90,8.8){$\Sigma_0$}
		\put(90,19.5){$\Sigma_{t_+}$}
		\put(50,9.5){$\operatorname{supp}(f)$}
		\put (16,9) {\textcolor[rgb]{0,.5,0}{$\operatorname{supp}(\psi_0)$}}
		\put (98,4) {$\partial M$}
		\put (98,11) {$\partial \Sigma_0$}
		\put (98,21.5) {$\partial \Sigma_{t_+}$}
		\put (-4,11) {$\partial \Sigma_0$}
		\put (-4,21.5) {$\partial \Sigma_{t_+}$}
		\put (-4,4) {$\partial M$}
	\end{overpic}
	\caption{Support of solution $\psi$.}
	\label{support:property:pic}
\end{figure}

\begin{proof}
	We will prove Statement~\eqref{support} only for the $J_+(\Sigma_0)$ case, the argument for $J_-(\Sigma_0)$ being completely analogous.
	We extend $(M,g)$ to a spatially compact globally hyperbolic spacetime $(\bar{M},\bar{g})$ without timelike boundary as described in Remark~\ref{bem:embedding}.
	Since $\bar{M}$ is obtained from $M$ be doubling, $\bar{M}$ is also spin and the Dirac operator extends as well.
	Since $\psi_0$ and $f$ are supported away from the boundary, we can extend them by zero to smooth sections $\bar{\psi}_0$ and $\bar{f}$ on $\bar{\Sigma}_0$ and $\bar{M}$, respectively.
	Here $\bar{\Sigma}_0$ is a smooth spacelike Cauchy hypersurface of $\bar{M}$ extending $\Sigma_0$.

	It is well known that the the Cauchy problem
	\begin{align*}
	\begin{cases}
	\bar{D}\bar{\psi}=\bar{f} \quad\text{on } \bar{M},& \\
	\bar{\psi}|_{\bar{\Sigma}_0}=\bar{\psi}_0,
	\end{cases}
	\end{align*}
	is well posed (see for example Section~3.7 in \cite{Waves}).
	Furthermore, its solution satisfies (see Corollary~3.7.5 in \cite{Waves}):
	\begin{align*}
	\operatorname{supp}(\bar{\psi})\cap J_\pm(\bar{\Sigma}_0)\subseteq J_\pm\big{(}\operatorname{supp}(\bar{f})\cap J_\pm(\bar{\Sigma}_0) \cup \operatorname{supp}(\bar{\psi}_0)\big{)}.
	\end{align*}
	Thus, for $t\in(t_-,t_+)$, the solution $\bar{\psi}$ is supported in the interior of $M$.
	By uniqueness, its restriction to $M$ must coincide with the solution $\psi$ of \eqref{CP:1} on $M_{(t_-,t_+)}$.
	This proves the statement on $M_{(t_-,t_+)}$.

	Now let $x\in M_{[t_+,\infty)}\setminus J_{\pm}\big{(}\operatorname{supp}(f)\cap J_{\pm}(\Sigma_0)\cup \operatorname{supp}(\psi_0)\cup \partial\Sigma_{t_{\pm}}\big{)}$.
	We need to show $\psi(x)=0$.
	Then $K:=J_-(x)\cap M_{[t_+,\infty)}$ does not intersect $\partial M$ nor the support of $f$, see Figure~\ref{support:property:pic2}.
	\begin{figure}[h]
		\begin{overpic}[scale=0.7]{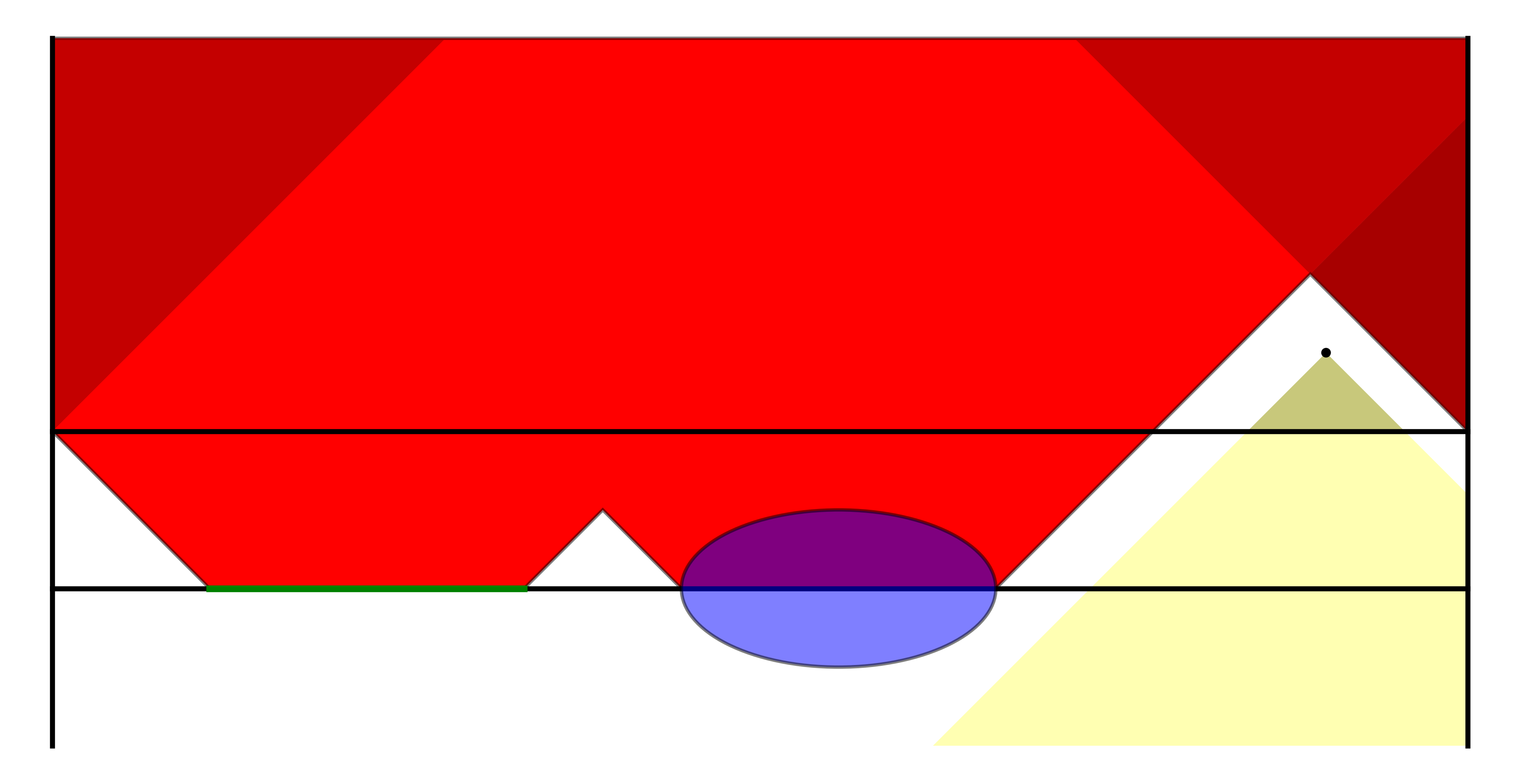}
			\put(49.7,9.5){$\operatorname{supp}(f)$}
			\put (16,9) {\textcolor[rgb]{0,.5,0}{$\operatorname{supp}(\psi_0)$}}
			\put (85,29){$x$}
			\put(86,24){$K$}
		\end{overpic}
		\caption{Region $K$ on which the energy estimate is applied.}
		\label{support:property:pic2}
	\end{figure}
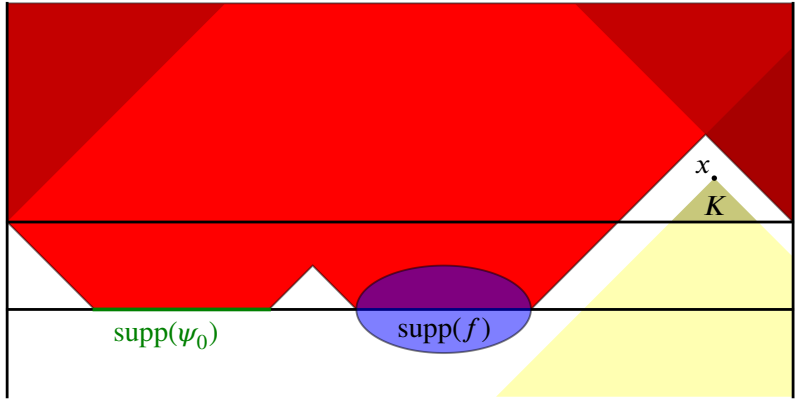

	Hence $D\psi=0$ on $K$.
	Since the spinorial Dirac operator is a symmetric hyperbolic system, we can apply the energy estimate Theorem~5.3 in \cite{BGreen}.
	This yields for all $t\in[t_+,t(x)]$
	\begin{align*}
	\int_{\Sigma_t\cap K} \abs{\psi}_0^2 \dt{\mu_{\Sigma_t}}
	\le
	C\left[\int_{t_+}^t\int_{\Sigma_s}\abs{D\psi}_0^2 \dt{\mu_{\Sigma_s}}\dt{s} +  \int_{\Sigma_{t_+}\cap K} \abs{\psi}_0^2 \dt{\mu_{\Sigma_{t_+}}}   \right]
	= 0.
	\end{align*}
	Thus $\psi$ vanishes on $K$ and, in particular, at $x$.
\end{proof}

For classical local elliptic boundary conditions in the sense of Lopatinsky-Shapiro, the term $\partial \Sigma_{t_\pm}$ in \eqref{support} is absent, see Proposition~3.3 in \cite{GM}.
The following example will show that we need it for nonlocal boundary conditions in general.
This means that a signal which hits the timelike boundary somewhere instantenously emanates across the boundary of the corresponding Cauchy hypersurface and radiates further starting from there.
This way, signals can spread at a speed bigger than that of light which raises questions about the physical interpretation.
One conclusion might be that nonlocal boundary conditions for the Dirac equation are simply unphysical.
However, in general, there are topological obstructions against local elliptic boundary conditions for the Dirac operator, see Subsection~II.7.B in \cite{Booss} or Section~II.6 in \cite{LM}.
When these are nontrivial, this would mean that there are no physical boundary conditions at all.

\begin{bsp}\label{trans:min}
	Let $(M,g)=(\mathbb{R}\times[0,1],g_\mathrm{Min}=-\dt{t}^2+\dt{x}^2)$ be a vertical strip in Minkowski space.
	It has timelike boundary $\partial M = \mathbb{R}\times\{0,1\}$.
	We identify the spinor bundle $SM$ with the trivial bundle $M\times\mathbb{C}^2$.
	We impose the Lorentzian transmission conditions discussed in Example~\ref{Lor:trans}.
	For simplicity, we consider the homogeneous initial value problem
	\begin{align}\label{CP:trans}
	\begin{cases}
	D\psi=0, & \\
	\psi(0,x)=\psi_0(x), & \\
	\psi(t,0)=\psi(t,1),
	\end{cases}
	\end{align}
	where $\psi_0(x)\in C^\infty_\c((0,1),\mathbb{C}^2)$ is a prescribed function.
	The solution is explicitly given by
	\[
	\psi(t,x)=\frac{ 1}{2}\sum_{k=-\infty}^\infty\left[\Mat{1,-1;-1,1}\psi_0(x+k+t)+\Mat{1,1;1,1}\psi_0(x+k-t)\right].
	\]
	Note that the sum is locally finite, so there is no convergence issue.
	Now, if the distance of the support of $\psi_0$ to $0$ is different from the one to $1$, then the term $\partial \Sigma_{t_\pm}$ in \eqref{support} cannot be dispensed with, see Figure~\ref{support:property:bsp}.
	\begin{figure}[h]
		\begin{overpic}[scale=0.7]{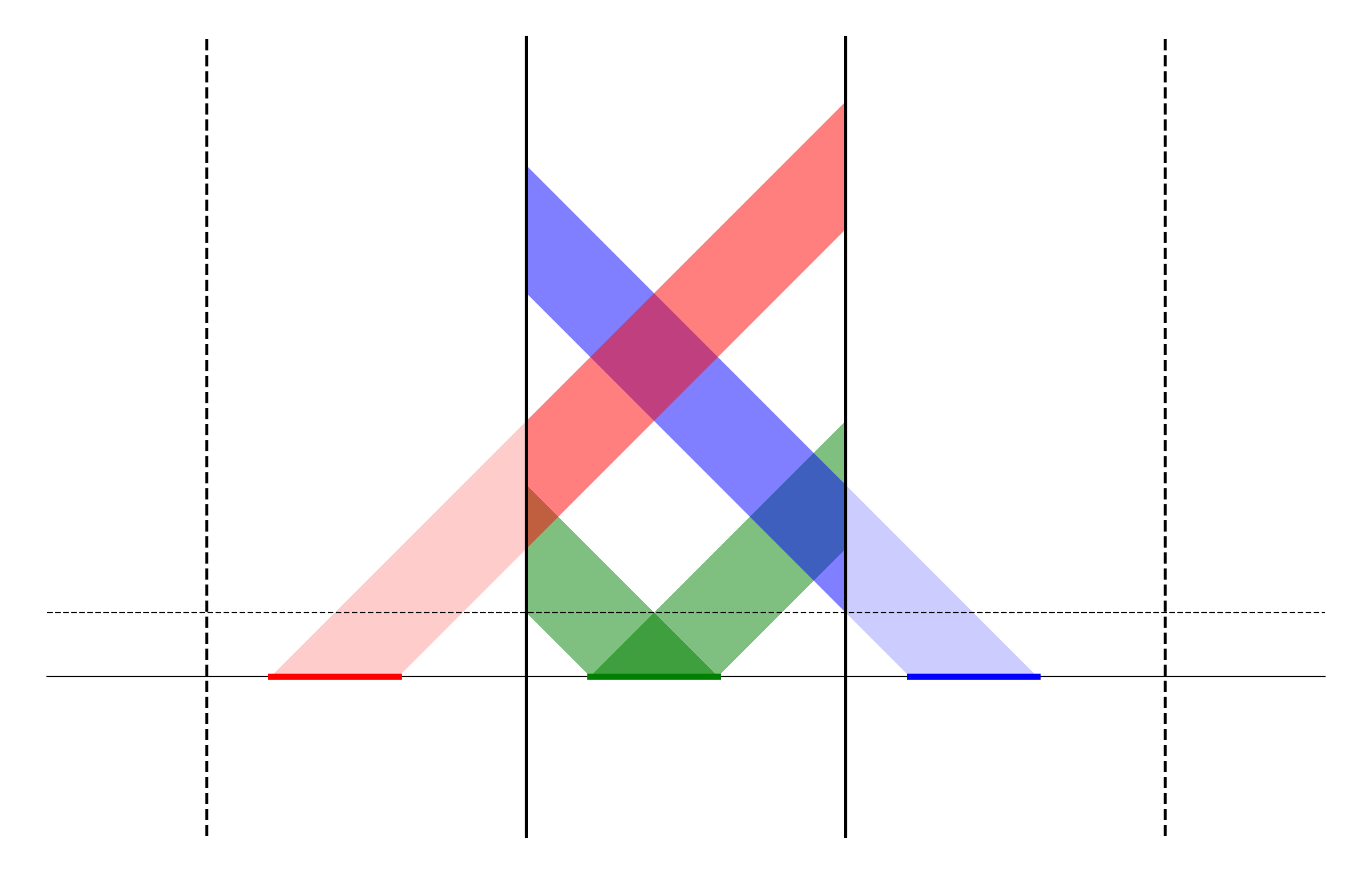}
		\put (42,11) {\textcolor[rgb]{0,.5,0}{$\operatorname{supp}(\psi_0)$}}
		\put(32.5,-1.9) {$\{x=0\}$}
		\put (56,-1.9) {$\{x=1\}$}
		\put (98,13.3) {$\{t=0\}$}
		\put (98,18) {$\{t=t_+\}$}
		\end{overpic}
		\caption{Support of solution to \eqref{CP:trans}.}
		\label{support:property:bsp}
	\end{figure}
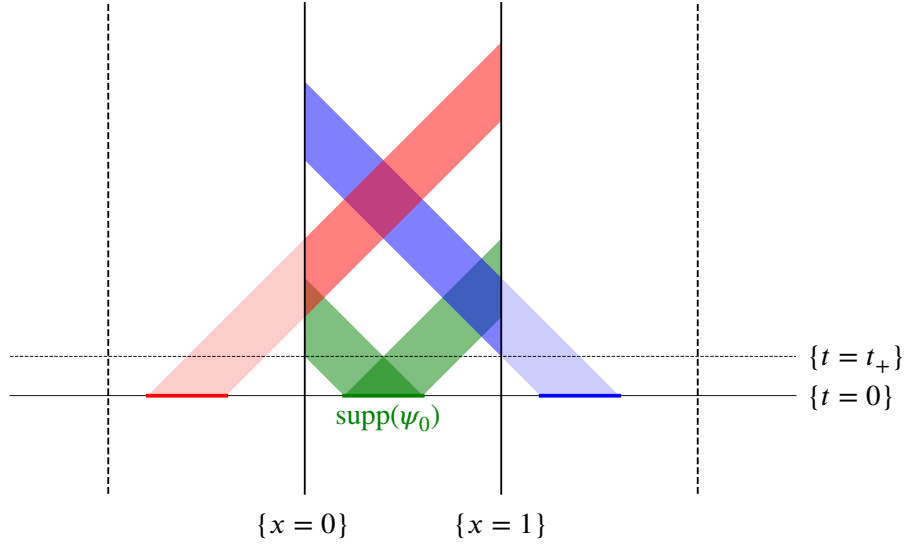
\end{bsp}

\subsection{Green's operators}

We next use the wellposedness of the Cauchy problem~\eqref{CP:1} together with Proposition~\ref{prop:support} to construct advanced and retarded Green's operators.
These operators depend on the choice of Cauchy temporal function and on the boundary conditions.
We set
\begin{align*}
C^\infty_{B,\cc}(&M,SM)\\
&:=
\big{\{}\psi\in C^\infty_\c(M,SM);~D\psi\in C^\infty_{\cc}(M,SM)~\text{and}~ \psi|_{\partial M}\in C^\infty(\partial M,B)\big{\}}.
\end{align*}

\begin{prop} \label{Green-func}
	Assume the Standard Setup~\ref{ssfirst}--\ref{sslast} and let $B=\{B_t\}_{t\in \mathbb{R}}$ be an admissible boundary condition.
	
	Then there exist unique linear maps $G_{\pm}\colon C^\infty_{\cc}(M,SM)\to C^\infty_{B,\cc}(M,SM)$ such that
	\begin{enumerate}[label=(\arabic*), labelwidth=!, labelindent=2pt, leftmargin=21pt]
		\item \label{Green:1} $DG_\pm=\id_{C^\infty_{\cc}}$,
		\item \label{Green:2} $G_\pm D=\id_{C^\infty_{B,\cc}}$, and
		\item \label{Green:3} $\operatorname{supp}(G_\pm f)\subseteq J_\pm(\operatorname{supp}(f) \cup \partial\Sigma_{t_\pm})$.
	\end{enumerate}
	Here
	\begin{align*} 
		t_+&=\min  T\big(J_+(\operatorname{supp}(f))\cap\partial M\big), \\
		t_-&=\max  T\big(J_-(\operatorname{supp}(f))\cap\partial M\big) .
	\end{align*}
\end{prop}
\begin{proof}
	We do the proof for $G_+$, the argument for $G_-$ being the same.
	\smallskip
	
	\emph{Existence.}
	Let $f\in C^\infty_{\cc}(M,SM)$.
	Choose $t'< T(\operatorname{supp}(f))$.
	Then $\operatorname{supp}(f)\subseteq I_+(\Sigma_{t'})$. 
	Now let $G_+f=\psi$ be the solution to the Cauchy problem
	\begin{align}\label{CP-Green-func}
	\begin{cases}
	D\psi=f\in C^\infty_{\cc}(M,SM), &\\
	\psi|_{\Sigma_{t'}}=0, & \\
	\psi|_{\partial M}\in C^\infty(\partial M, B).
	\end{cases}
	\end{align}
	Then Proposition~\ref{prop:support} implies that 
	\[
	\operatorname{supp}(G_+ f)\subseteq J_+(\operatorname{supp}(f)\cup \partial\Sigma_{t_+}).
	\]
	In particular, if $t'' < T(\operatorname{supp}(f))$, then $(G_+f)|_{\Sigma_{t''}}=0$.
	Hence one can replace $t'$ by $t''$ in \eqref{CP-Green-func}, i.e.\ the definition of $G_+$ is independent of the choice of $t'$.
	
	Property~\ref{Green:1} is satisfied by construction, so it is only left to show Property~\ref{Green:2}.
	Let $\tilde{\psi}\in C^\infty_{B,\cc}(M,SM)$ and put $f:=D\tilde{\psi}$.
	Choose $t'< T(\operatorname{supp}(\tilde{\psi}))$.
	Then we also have $t'< T(\operatorname{supp}(f))$.
	
	Now, $\psi=G_+ f$ is the unique solution to the Cauchy problem~\eqref{CP-Green-func} with $\psi|_{\Sigma_{t'}}=0$.
	Thus we have $D\psi=f=D\tilde{\psi}$, $\psi|_{\Sigma_{t'}}=\tilde{\psi}|_{\Sigma_{t'}}=0$, and $\psi|_{\partial M},\tilde{\psi}|_{\partial M}\in  C^\infty(\partial M, B)$.
	Therefore $\tilde{\psi}-\psi$ solves the homogeneous Cauchy problem~\eqref{CP-Green-func} with $f=0$.
	By uniqueness, $\tilde{\psi}-\psi=0$, i.e.\ $G_+D\tilde{\psi}=G_+f=\psi=\tilde{\psi}$.
	
	\smallskip
	\emph{Uniqueness.}
	Let $G_+, \tilde{G}_+\colon C^\infty_{\cc}(M,SM)\to C^\infty_{B,\cc}(M,SM)$ be two linear operators satisfying \ref{Green:1}--\ref{Green:3}.
	Let $f\in C^\infty_{\cc}(M,SM)$.
	Then $\psi:=G_+f - \tilde{G}_+f$ solves the homogeneous Cauchy problem~\eqref{CP-Green-func} for any $t'<T(\operatorname{supp}(f))$.
	Thus $\psi=0$, i.e.\ $G_+f = \tilde{G}_+f$.
\end{proof}

Following the usual terminology, we call $G_+$ the \emph{retarded Green's operator} and $G_-$ the \emph{advanced Green's operator}.

\section{Pseudo Local Boundary Conditions}
\label{sec:examples}

In this section we focus on pseudolocal boundary conditions and see how to ensure the properties required in Definition~\ref{admissibleBound}.
We will first consider Grassmannian projections as introduced in Definition~\ref{Grassmannian:Def} and then apply this to the famous Atiyah-Patodi-Singer conditions.

\subsection{Grassmannian projections}

We start by showing that Grassmannian projections that have a few natural properties lead to admissible boundary conditions.

\begin{satz}\label{main:Grassmannian}
	Assume the Standard Setup~\ref{ssfirst}--\ref{sslast}.
	Let $\{P_t\}_{t\in\mathbb{R}}$ be a family of selfadjoint pseudodifferential projections on $L^2(\partial\Sigma_t,SM|_{\partial\Sigma_t})$ such that
	\begin{enumerate}[label=(\alph*), labelwidth=!, labelindent=2pt, leftmargin=21pt]
		\item \label{Grass:1} $N^{\nicefrac 12}(t,\cdot)P_t=P_t N^{\nicefrac 12}(t,\cdot)$,
		\item \label{Grass:2} $P_t=\id+\sigma_{D_t}(\eta_t^\flat)P_t \sigma_{D_t}(\eta_t^\flat)$,
		\item \label{Grass:3} $\tilde{P}_t:=U(t) P_t  U(t)^{-1}$ is a Grassmannian projection on $L^2(\hat{\Sigma}_0,SM|_{\partial\hat{ \Sigma}_0})$, and
		\item \label{Grass:4} $t\mapsto \tilde{P}_t$ is $H^{s}$-norm continuous for every $s\in\mathbb{N}_0$.
	\end{enumerate}
	Then $B:=\{B_t\}_{t\in\mathbb{R}}$ with $B_t:=P_tH^\frac{1}{2}(\partial\Sigma_t,SM|_{\partial\Sigma_t})$ is an admissible boundary condition.
\end{satz}
\begin{proof}
	Assumption~\ref{Grass:2} and $P_t=P_t^*$ imply by Corollary~\ref{Grassmannian:cor1} that $B_t$ is selfadjoint for $D_t$.
	By Assumption~\ref{Grass:1}, $B_t=N(t)^{\nicefrac 12}B_t$ which yields \ref{admissibleBound1} in Definition~\ref{admissibleBound}.
	Assumption~\ref{Grass:3} implies by Corollary~\ref{Grassmannian:cor2} that $U(t)B_t$ is a selfadjoint $\infty$-regular boundary condition for $\tilde{D}_t$.
	This is \ref{admissibleBound2} in Definition~\ref{admissibleBound} because $U(t)B_t = U(t)N(t)^{\nicefrac n2}B_t=\tilde{B}_t$.
	Furthermore, Assumption~\ref{Grass:4} implies Assertion~\ref{admissibleBound3} in Definition~\ref{admissibleBound} by Lem\-ma~\ref{HkContinuity}.
	\end{proof}
Under the assumptions of Theorem~\ref{main:Grassmannian}, Theorem~\ref{main} implies the well-posedness of Cauchy problem~\eqref{CP:1}.
\begin{bem}
	Condition~\ref{Grass:4} can be replaced by the weaker requirement that $t\mapsto \tilde{P}_t$ is just $L^2$-norm continuous provided we can choose the boundary operators $\tilde{A}_t$ for $\tilde{D}_t$ such that they commute or anticommute with $\tilde{P}_t$.
	In this case, Lemma~\ref{APS-continuity} ensures that $t\mapsto \tilde{D}_{t,\tilde{B}_t}J^{(\varepsilon)}_{\tilde{B}_t}$ is $H^k$-norm continuous for all $k\in\mathbb{N}_0$.
\end{bem}

By adding additional assumptions on the geometry of the spacetime, we can simplify the assumptions on the pseudodifferential projections as follows:
\begin{koro}\label{koro1:Grassmannian}
	Assume the Standard Setup~\ref{ssfirst}--\ref{sslast}.
	Let $\{P_t\}_{t\in\mathbb{R}}$ be a family of selfadjoint Grassmannian projections on $L^2(\partial\Sigma_t,SM|_{\partial\Sigma_t})$ such that $t\mapsto \tilde{P}_t=U(t)P_tU(t)^{-1}$ is $H^{s}$-norm continuous for all $s\in\mathbb{N}_0$.

	If, in addition, $\eta$ is parallel along the integral curves of $\partial_t$ and $N|_{\partial M}$ depends only on $t$, then $B=\{B_t\}_{t\in\mathbb{R}}$ with $B_t=P_tH^\frac{1}{2}(\partial\Sigma_t,SM|_{\partial\Sigma_t})$ is an admissible boundary condition.
\end{koro}

\begin{proof}
	Since $N|_{\partial M}$ depends only on $t$, Condition~\ref{Grass:1} in Theorem~\ref{main:Grassmannian} holds.
	Condition~\ref{Grass:2} is part of the definition of a Grassmannian projection and Condition~\ref{Grass:4} is assumed.
	Since $\eta$ is parallel along the integral curves of $\partial_t$, we see that $\tilde{P}_t$ is a Grassmannian projection as well, which is Condition~\ref{Grass:3}.
\end{proof}

\subsection{Atiyah-Patodi-Singer conditions}

We now use the previous section to discuss the Atiyah-Patodi-Singer boundary conditions.
We assume the Standard Setup \ref{ssfirst}--\ref{sslast} and, additionally, that the boundary operators $A_t$ have trivial kernel.
As noted in Example~\ref{bsp_elliptic}, the Atiyah-Patodi-Singer conditions
\[
B_{\APS,t}:=\chi^-(A_t)H^{\frac 12}(\partial\Sigma,SM|_{\partial\Sigma_t})
\]
are selfadjoint boundary conditions for $D_t$.
We obtain the following well-posedness result:

 \begin{koro}\label{main-APS}
 	Assume the Standard Setup~\ref{ssfirst}--\ref{sslast} and additionally assume that $N|_{\partial M}$ depends only on $t$, that $\eta$ is parallel along the integral curves of $\partial_t$ and that the kernel of $A_t$ is trivial for every $t\in\mathbb{R}$.

 	Then, given $f\in C^\infty_{\cc}(M,SM)$ and $\psi_0\in C^\infty_{\cc}(\Sigma_0,SM|_{\Sigma_0})$ there exists a unique smooth solution $\psi\in C^\infty(M,SM)$ to
 	\begin{align*} 
 		\begin{cases}
 			D\psi=f, &\\
 			\psi|_{\Sigma_0}=\psi_0, & \\
 			\psi|_{\partial M}\in C^\infty(\partial M, B_{\APS}).
 		\end{cases}
 	\end{align*}
 	The solution depends continuously on the Cauchy data $(f,\psi_0)$.
 \end{koro}
 
\begin{proof}
Since $N|_{\partial \Sigma_t}$ is constant, $N(t)^{\nicefrac 12}B_{\APS,t}=B_{\APS,t}$ is a selfadjoint $\infty$-regular boundary condition for $D_t$.
Thus $\tilde{B}_t=U(t)N^{\nicefrac n2}(t)B_{\APS,t} = U(t)B_{\APS,t}$ is a selfadjoint $\infty$-regular boundary condition for $\tilde{D}_t=N(t)U(t)D_tU(t)^{-1}$.
By Remark~\ref{boprincipalsymbol}, we can choose the boundary operator $\tilde{A}_t$ for $\tilde{D}_t$ as
\[
\tilde{A}_t=N(0)U(t)A_tU(t)^{-1}.
\]
Since $\chi^-(A_t)$ commutes with $A_t$, this is also true for $\tilde{A}_t$ and $\tilde{P}_t$.
Now Lemma~\ref{APS-continuity} together with Remark~\ref{bem-APS-cont} shows that $t\mapsto \tilde{P}_t$ is $H^s$-norm continuous for every $s\in\mathbb{N}_0$.
Corollary~\ref{koro1:Grassmannian} implies that $\{B_{\APS,t}\}_t$ is an admissible boundary condition and Theorem~\ref{main-APS} concludes the proof.
\end{proof}

Corollary~\ref{main-APS} is a slight generalization of the result stated in \cite{DGM}.

\end{document}